\documentclass[a4paper]{amsart}
\usepackage[
bookmarks=true,
colorlinks=true,
linkcolor=blue,
urlcolor=blue,
bookmarks=true,
citecolor=blue,
linktocpage=true,
hyperindex=true
]{hyperref}
\usepackage{color}
\usepackage{amssymb,amsmath,amsthm,amstext,amsfonts}
\usepackage[dvips]{graphicx}
\usepackage{psfrag}
\usepackage{url}
\usepackage{amsfonts}
\usepackage{amsmath}
\usepackage{mathrsfs}
\usepackage{graphicx}
\usepackage{amsmath,amsthm,amssymb,amscd}
\usepackage{enumerate}

\usepackage{epsfig}

\makeatletter \@addtoreset{equation}{section} \makeatother

\renewcommand\thetable{\thesection.\@arabic\c@table}

\theoremstyle{plain}
\newtheorem{maintheorem}{Theorem}

\newtheorem{maincorollary}{Corollary}

\newtheorem{mainproposition}{Proposition}

\newtheorem{mainlemma}{Lemma}

\newtheorem{theorem}{Theorem}[section]
\newtheorem{lemma}[theorem]{Lemma}

\newtheorem{Claim}[theorem]{Claim}

\newtheorem{definition}[theorem]{Definition}
\newtheorem{remark}[theorem]{Remark}

\theoremstyle{remark}

\makeatletter
\@namedef{subjclassname@2020}{2020 Mathematics Subject Classification}
\makeatother

\begin{document}

\title{On the loss of upper semicontinuity of metric entropy for $C^{r}$ diffeomorphisms}

\author{Xinyu Bai, Wanshan Lin, and Xueting Tian}

\address{Xinyu Bai, School of Mathematical Sciences,  Fudan University\\Shanghai 200433, People's Republic of China}
\email{25110180001@m.fudan.edu.cn}

\address{Wanshan Lin, School of Mathematical Sciences,  Fudan University\\Shanghai 200433, People's Republic of China}
\email{wanshanlin@fudan.edu.cn}

\address{Xueting Tian, School of Mathematical Sciences,  Fudan University\\Shanghai 200433, People's Republic of China}
\email{xuetingtian@fudan.edu.cn}

\begin{abstract}
In this article, we give an upper bound estimate for the quantitative loss of upper semicontinuity of metric entropy for $C^r\:(r>1)$ diffeomorphisms. Building on earlier entropy estimates and reparametrization methods, we optimize the upper bound estimate with respect to both dimension and asymptotic Lipschitz constant. Motivated by examples of Newhouse and Buzzi, we show that the estimate is sharp.
\end{abstract}

\maketitle
\section{Introduction}
\subsection{Background and motivation}
The theory of differentiable dynamical systems is concerned with the complex and chaotic behavior of diffeomorphisms on Riemannian manifolds. Metric entropy is a fundamental concept in ergodic theory, which was introduced by Kolmogorov and Sinai. Let $h_{\mu}(f)$ be metric entropy for a measurable map $f$ with respect to an $f$-invariant measure $\mu$.

In general, metric entropy is not lower semicontinuous with respect to invariant measures, even for uniformly hyperbolic diffeomorphisms. On the other hand, upper semicontinuity of metric entropy is fundamental, since it provides a technical sufficient condition for the existence of measures of maximal entropy and equilibrium states. Moreover, it helps to control the variation of entropy under perturbations, thus underpinning robust statistical descriptions of chaotic systems.

To establish upper semicontinuity of metric entropy, one usually works in settings involving diffeomorphisms with certain hyperbolic structures or smoothness. In \cite{CY16,P12,LSW15,L13}, for diffeomorphisms with certain hyperbolic structures, upper semicontinuity of metric entropy is guaranteed by the expansiveness-like property via the results of \cite{B72,N89}, based on the regular geometric structure of the stable and unstable manifolds. In \cite{B97,N89,Y87}, for $C^{\infty}$ diffeomorphisms, upper semicontinuity of metric entropy is ensured by reparametrization methods. In contrast, in \cite{B14,LSW15}, for $C^{r}\:(1<r<+\infty)$ diffeomorphisms, upper semicontinuity of metric entropy may sometimes fail to hold.  

This leads to the following fundamental question: under what conditions is metric entropy upper semicontinuous for $C^r$ diffeomorphisms with $1<r<+\infty$? This question has attracted considerable attention in recent years. For $C^r$ ($r>1$) surface diffeomorphisms, Burguet \cite{B23} established upper semicontinuity of metric entropy specifically at ergodic measures with large entropy. For $C^r$ ($r>1$) diffeomorphisms on a $d$-dimensional compact Riemannian manifold with $d\le 3$, Luo and Yang \cite{LY25-2} proved that the continuity of the sum of positive Lyapunov exponents implies upper semicontinuity of metric entropy.  

A related natural question is how much entropy can be lost when upper semicontinuity fails. By \cite{B97,N89,Y87}, for a $C^r\:(r>1)$ diffeomorphism $f$ and a sequence of $f$-invariant measures $\left\{\mu_n\right\}_{n\in\mathbb N^+}$ converging to an $f$-invariant measure $\mu$ with respect to the weak* topology, one has
$$\limsup_{n\to+\infty}h_{\mu_n}(f)\le h_{\mu}(f)+\frac{d\min\left\{\lambda^+(f),\lambda^+(f^{-1})\right\}}{r},$$
where $\lambda^+(f)=\lim_{n\to+\infty}\frac{1}{n}\log^+\left\|Df^n\right\|_{0}$ denotes the asymptotic logarithmic Lipschitz constant. By \cite{B12}, for a $C^r\:(r>1)$ surface diffeomorphism $f$ and a sequence of $f$-invariant measures $\left\{\mu_n\right\}_{n\in\mathbb N^+}$ converging in the weak* topology to an $f$-invariant measure $\mu$, one has
$$\limsup_{n\to+\infty}h_{\mu_n}(f)\le h_{\mu}(f)+\frac{\min\left\{\lambda^+(f),\lambda^+(f^{-1})\right\}}{r}.$$ 

For a $C^r$ diffeomorphism $f$,
we define the quantitative loss of upper semicontinuity of metric entropy with respect to an $f$-invariant measure $\mu$ by:
\begin{equation*}
\begin{aligned}
e_{\mu}(f):=\sup\left\{ \limsup_{n\to+\infty}h_{\mu_n}(f_n)-h_{\mu}(f):f_n\xrightarrow{C^r}f,\:\mu_n\in\mathcal M(\mathbf M^d,f_n),\:\mu_n\to\mu\:\mathrm{as}\:n\to+\infty\right\}.
\end{aligned}
\end{equation*}
 In this paper, we derive a sharper estimate for $e_{\mu}(f)$ than the previously known bounds in \cite{B12,B97,N89,Y87}, for all $C^r\:(r>1)$ diffeomorphisms $f$, for all $f$-invariant measures $\mu$. We also show that our estimate is optimal for certain examples, for all $r>1$ and $d>1$. Moreover, when $d=2$, these examples form a relatively general class.

\subsection{Settings and results}
Throughout this paper, $(X,\mathrm d)$ denotes a compact metric space, and  $T:X\to X$ denotes a homeomorphism.
Let $\mathcal M(X),\:\mathcal M(X,T),\:\mathcal M^{erg}(X,T)$ be the sets of all Borel probability measures, $T$-invariant Borel probability measures, $T$-invariant and ergodic Borel probability measures on $X$ respectively. Let $\mathbb N^+$ be the set of all positive integers, and let $\mathbb R^{+}$ be the set of all non-negative real numbers. Assume that $\left\{T_n\right\}_{n\in\mathbb N^+}$ is a sequence of homeomorphisms on $X$, then a sequence $\left\{\mu_n\right\}_{n\in\mathbb N^+}$ is called a sequence of
$\left\{T_n\right\}_{n\in\mathbb N^+}$-invariant measures if
$$\forall\:n\in\mathbb N^+,\mu_n\in\mathcal M(X,T_n).$$
Let $C(X)$ be the set of all continuous functions on $X$. We define the weak* metric $\rho(\cdot,\cdot)$ in $\mathcal M(X)$ by
$$\rho(.,.):\mathcal M(X)\times\mathcal M(X)\to\mathbb R^+,\:(\mu,\nu)\mapsto\sum_{i\in\mathbb N^+}\frac{1}{2^i}\left|\int g_id\mu-\int g_id\nu\right|,$$ where $\left\{g_i\right\}_{i\in\mathbb N^+}$ is a dense subset of the unit sphere of $C(X)$. With the metric $\rho$, both $\mathcal M(X)\:\:\mathrm{and}\:\:\mathcal M(X,T)$ are compact metric spaces. We say that a sequence of Borel probability measures $\left\{\mu_n\right\}_{n\in\mathbb N^+}$ converging to a measure $\mu$, if $$\lim_{n\to+\infty}\rho(\mu_n,\mu)=0.$$ 

Let $\mathbf M^d$ be a compact $d$-dimensional Riemannian manifold. For any real number $r\ge1$, $\mathrm{Diff}^r(\mathbf M^d)$ denotes the set of all $C^r$ diffeomorphisms on $\mathbf M^d$. For $f\in\mathrm{Diff}^r(\mathbf M^d)$,  let $\Gamma_{f}$ be the set of points that are regular in the sense of Oseledets. For $x\in\Gamma_{f}$, let
$$\lambda_1(f,x)>\lambda_2(f,x)>\cdots>\lambda_{\ell(f,x)}(f,x)$$
denote its distinct Lyapunov exponents and let
$$\mathrm T_x\mathbf M^d=E^{1}_{f}(x)\oplus\cdots\oplus E^{\ell(f,x)}_{f}(x)$$
be the corresponding decomposition of its tangent space. If $\nu\in \mathcal{M}^{erg}(\mathbf M^d,f)$, the functions $$x\mapsto \ell(f,x),\:x\mapsto\lambda_i(f,x),\:x\mapsto\mathrm{dim}(E^i_{f}(x))$$ 
are constant for $\nu$-$\mathrm{a.e.}\:x\in\mathbf M^d$.

For $i\in\mathbb N^+$ and $\nu\in\mathcal M(\mathbf M^d,f)$, let $$\lambda_i(f,\nu)=\int \lambda_i(f,x)d\nu(x),$$
whenever $\lambda_i(f,x)$ exists for $\nu$-a.e.\:$x\in\mathbf M^d$. 
For $x\in\Gamma_{f}$, let $$E^{u}_{f}(x)=\bigoplus_{\lambda_i(f,x)>0}E^i_{f}(x),$$
and $$d_{\max}^u(f,\mu)=\mathrm{ess\,sup}_{\mu\text{-a.e.\:}x\in\mathbf M^d}\mathrm{dim}E^u_{f}(x):=\min_{A\subset O(\mu),\:\mu(A)=1}\max_{x\in A}\mathrm{dim}E^u_f(x),$$ where $O(\mu)$ denotes the regular set in the sense of Oseledets with respect to $\mu\in\mathcal M(\mathbf M^d,f)$.
For $d\in\mathbb N^+$, $f\in\mathrm{Diff}^r(\mathbf M^d)$ and $\mu\in\mathcal M(\mathbf M^d,f)$, denote $$\lambda_1^+(f,\mu)=\int\lambda_1^+(f,x)d\mu(x),$$  
$$\lambda^+_{\mathrm{max}}(f,\mu)=\left\{\begin{matrix}
 \min\left\{\lambda^+_{1}(f,\mu),\lambda^+_{1}(f^{-1},\mu)\right\} & d\le 2\\
 \int\max\left\{\lambda^+_{1}(f,x),\lambda^+_{1}(f^{-1},x)\right\}d\mu(x)  & d>2
\end{matrix}\right.,$$
 where
$$\lambda_1^+(f,x)=\max\left\{\lambda_1(f,x),0\right\},$$
if $\lambda_1(f,x)$ exists.
\begin{maintheorem}\label{Theorem A}
Let $(f_n)_{n\in\mathbb N^+}$ be a sequence of $C^r$ diffeomorphisms converging $C^r$ to $f\in\mathrm{Diff}^{r}(\mathbf M^d)$ with $r>1$. Assume there is a sequence $\left\{\mu_n\right\}_{n\in\mathbb N^+}$ of $\left\{f_n\right\}_{n\in\mathbb N^+}$-invariant measures converging to an $f$-invariant measure $\mu$. Then, one has 
\begin{equation}
\begin{aligned}\label{(1.1)}
\limsup_{n\to+\infty}h_{\mu_n}(f_n)\le h_{\mu}(f)+\limsup_{n\to+\infty}d_{\max}^u(f_n,\mu_n)\frac{\lambda_{1}^+(f,\mu)}{r}.
\end{aligned}
\end{equation}
Moreover, if\: $\limsup_{n\to+\infty}d^u_{\max}(f_n,\mu_n)=1$, one has
\begin{equation}\label{1.2}
\begin{aligned}
\limsup_{n\to+\infty}h_{\mu_n}(f_n)\le h_{\mu}(f)+\frac{1}{r-1}(\lambda_1^+(f,\mu)-\liminf_{n\to+\infty}\lambda_1^+(f_n,\mu_n)).
\end{aligned}
\end{equation}
\end{maintheorem}
Even in the case \(d=2\), Theorem \ref{Theorem A} provides the sharpest available estimate for the quantitative loss in the upper semicontinuity of metric entropy, improving the results in \cite{B12,B97,N89,Y87}. Compared with \cite{B97,N89,Y87}, our upper bound for \(e_\mu(f)\) replaces the ambient dimension \(d\) by
\[
\limsup_{n\to+\infty} d^u_{\max}(f_n,\mu_n).
\]
 Moreover, we replace the asymptotic Lipschitz constant \(\lambda^+(f)\) by the positive part of the maximal Lyapunov exponent of \(\mu\) with respect to \(f\). Neither of these refinements can be obtained using the techniques and strategies developed in \cite{B12,B97,N89,Y87}.

Under the assumptions of Theorem A, when \(d\le 3\) and
\[
\limsup_{n\to+\infty} d^u_{\max}(f_n,\mu_n)=1,
\]
formula \eqref{1.2} was proved by Luo and Yang in \cite{LY25-2}. A similar result is implicit in \cite{BL21}, while the case \(d=1\) was established in \cite{B17}. 

 Corollary \ref{Corollary A} is a direct consequence of Theorem \ref{Theorem A}. It makes the dimensional improvement in Theorem \ref{Theorem A} more transparent and clarifies its consequences for the quantitative loss in the upper semicontinuity of topological entropy. The key observation is that every ergodic component of an invariant measure either has unstable manifolds of dimensions less than or equal to \([\frac{d}{2}]\) or stable manifolds of dimensions less than or equal to \([\frac{d}{2}]\).
\begin{maincorollary}\label{Corollary A}
Let $(f_n)_{n\in\mathbb N^+}$ be a sequence of $C^r$ diffeomorphisms converging $C^r$ to $f\in\mathrm{Diff}^{r}(\mathbf M^d)$ with $r>1$. Assume there is a sequence $\left\{\mu_n\right\}_{n\in\mathbb N^+}$ of $\left\{f_n\right\}_{n\in\mathbb N^+}$-invariant measures converging to an $f$-invariant measure $\mu$. Then, one has
\begin{equation}
\limsup_{n\to+\infty}h_{\mu_n}(f_n)\le h_{\mu}(f)+[\frac{d}{2}]\frac{\lambda^+_{\max}(f,\mu)}{r},
\end{equation}
\begin{equation}
\limsup_{n\to+\infty}h_{\mathrm{top}}(f_n)\le h_{\mathrm{top}}(f)+[\frac{d}{2}]\frac{\sup_{\mu\in\mathcal M(\mathbf M^d,f)}\lambda^+_{\max}(f,\mu)}{r}.
\end{equation}
In other word, 
$$e_{\mu}(f)\le [\frac{d}{2}]\frac{\lambda^+_{\max}(f,\mu)}{r}.$$
\end{maincorollary}

For $\Lambda\subset\mathbf M^d$, let
$$\mathcal M(\Lambda,f)=\left\{\mu\in\mathcal M(\mathbf M^d,f):\mu(\Lambda)=1\right\}.$$
Assume that $\Lambda$ is $f$-invariant, that is, $f(\Lambda)=\Lambda$. A splitting $T_{\Lambda}\mathbf M^d=E_1\oplus E_2$ is said to be measurable and $f$-invariant, if the following hold:
$$\Lambda\ni x\mapsto E_i(x)\:\mathrm{is\:measurable},\:i=1,2,$$
$$D_{x}f(E_i(x))=E_i(fx),\:\forall\:x\in\Lambda.$$
Corollary \ref{Corollary B} follows from \eqref{1.2} in Theorem \ref{Theorem A}. In contrast to the results in \cite{CY16,LSW15,LMZ24}, we do not require the diffeomorphism to admit a dominated splitting. Instead, we impose stronger assumptions on the dimension of its unstable bundle.

\begin{maincorollary}\label{Corollary B}
Assume that $f$ is a $C^r$ ($r>1$) diffeomorphism on $\mathbf M^d$, and let $\Lambda\subset \mathbf M^d$ be an $f$-invariant set on which there exists a measurable $f$-invariant splitting $T_{\Lambda}\mathbf M^d = E^{cu} \oplus E^{cs}$, such that\\
$\bullet$ for any $z\in\Lambda$, $\mathrm{dim}E^{cu}(z)=1$;\\
$\bullet\:z\mapsto E^{cu}(z)$ is continuous on $\Lambda$;\\
$\bullet$ for any $z\in\Lambda$,  one of the following two alternatives holds:\\
(1)
$$\liminf_{n\to\pm\infty}\frac{1}{n}\log\left\|D_zf^n|_{E^{cu}(z)}\right\|\ge0,$$
$$\limsup_{n\to\pm\infty}\frac{1}{n}\log\left\|D_zf^n|_{E^{cs}(z)}\right\|<0,$$
(2)
$$\liminf_{n\to\pm\infty}\frac{1}{n}\log\left\|D_zf^n|_{E^{cu}(z)}\right\|>0,$$
$$\limsup_{n\to\pm\infty}\frac{1}{n}\log\left\|D_zf^n|_{E^{cs}(z)}\right\|\le0.$$
Assume there is a sequence $\left\{\mu_n\right\}_{n\in\mathbb N^+}$ of measures with 
$$\forall\:n\in\mathbb N^+,\:\mu_{n}\in\mathcal M(\Lambda,f),$$ 
converging to a measure $\mu\in\mathcal M(\Lambda,f)$. Then, one has
\begin{equation*}
\limsup_{n\to+\infty}h_{\mu_n}(f)\le h_{\mu}(f).
\end{equation*}

\end{maincorollary}
Corollary \ref{Corollary C} follows from the classical fact that the set of continuity points of an upper semicontinuous function on a compact metric space is residual.
\begin{maincorollary}\label{Corollary C}
Assume that $f\in\mathrm{Diff}^r(\mathbf M^d)\:(r>1)$, such that for any $\mu\in\mathcal M^{erg}(\mathbf M^d,f)$, one of the following two alternatives holds:\\
$\bullet\:$$\mu$ has exactly $0\:\mathrm{or}\:1$ positive Lyapunov exponent;\\
$\bullet\:$$\mu$ has exactly $0\:\mathrm{or}\:1$ negative Lyapunov exponent.\\
Then there exists a residual subset $\mathcal M$ of $\mathcal M(\mathbf M^d,f)$, such that for any $\nu\in\mathcal M$, the entropy map $\mu\mapsto h_{\mu}(f)$ is upper semicontinuous at $\nu$.
\end{maincorollary}
Similar to the proof of \cite[Theorem 6.1]{DN05} and \cite{LSW15}, one has, in the Newhouse domain, there exists an abundance of diffeomorphisms admitting an ergodic measure whose loss of the upper semicontinuity of metric entropy attains the critical threshold. Let's first introduce the basic concepts of the Newhouse domain.

Let \(\Lambda\) be a hyperbolic basic set with respect to $f\in\mathrm{Diff}^r(\mathbf M^2)$. A non-degenerate tangency point
\(q\) with respect to the unstable manifold \(W^u(f,x)\) at $x$ and the stable manifold \(W^s(f,y)\) at $y$, for some \(x,y\in \Lambda\), is called a
non-degenerate homoclinic tangency associated with \(\Lambda\) of \(f\). We say that \(\Lambda\) is a wild hyperbolic set if it exhibits a persistent
homoclinic tangency. More precisely, there exists a neighborhood
\(\mathcal U(f)\subset \mathrm{Diff}^r(\mathbf M^2)\) of \(f\) such that, for every
\(g\in \mathcal U(f)\), the continuation \(\Lambda_g\) of \(\Lambda\) has a
non-degenerate homoclinic tangency. In \cite[Theorem 1]{N79}, it has been proved that if \(f\) has a hyperbolic basic set \(\Lambda\) with a homoclinic
tangency \(q\), \(f\) can be perturbed into an open set
\(\mathcal N\subset \mathrm{Diff}^r(\mathbf M^2)\), \(r\ge 2\), such that every
\(g\in \mathcal N\) has a wild hyperbolic set near
\(\operatorname{Orb}(q,f)\).
The open set \(\mathcal N\) is usually called a Newhouse domain, or a
Newhouse open set.

\begin{maintheorem}\label{Corollary E}
Let $2\le r<+\infty$ and $\mathcal N\subset \mathrm{Diff}^r(\mathbf M^2)$ be a Newhouse domain. Then there exists a residual subset $\mathcal R\subset \mathcal N$ such that for any $f\in\mathcal R$, there exist $\mu\in\mathcal M^{erg}(\mathbf M^2,f)$ and a sequence of $f$-invariant and ergodic measures $\left\{\mu_n\right\}_{n\in\mathbb N^+}$ with $\mu_n\to\mu$, as $n\to+\infty$, such that
$$\lim_{n\to+\infty}h_{\mu_n}(f)=h_{\mu}(f)+\frac{\lambda^+_{\max}(f,\mu)}{r}.$$
\end{maintheorem}
Part (1) of Theorem \ref{Theorem B} establishes the sharpness of the estimate in \eqref{(1.1)} of Theorem \ref{Theorem A}. More precisely, for every \(r>1\) and \(d>1\), there exist \(f\in \mathrm{Diff}^r(\mathbf M^d)\) and \(\mu\in \mathcal M(\mathbf M^d,f)\) for which the bound in \eqref{(1.1)} is attained. Part (2) of Theorem \ref{Theorem B} shows that, in some cases, the critical values in the estimates \eqref{(1.1)} and \eqref{1.2} of Theorem \ref{Theorem A} can be attained simultaneously.
\begin{maintheorem}\label{Theorem B}
(1) For any $d>1$, there exists a compact $d$-dimensional Riemannian manifold $\mathbf M^d$,
for any $1<r<+\infty$, there exists a sequence $\left\{f_{n}\right\}_{n\in\mathbb N^+}$ of $C^r$ diffeomorphisms on $\mathbf M^d$ converging $C^r$ to $f\in\mathrm{Diff}^\infty(\mathbf M^d)$ and a sequence $\left\{\mu_{n}\right\}_{n\in\mathbb N^+}$ of $\left\{f_{n}\right\}_{n\in\mathbb N^+}$-invariant measures converging to an $f$-invariant measure $\mu$, such that
$$\lim_{n\to+\infty}h_{\mu_{n}}(f_{n})= h_{\mu}(f)+\limsup_{n\to+\infty}d^u_{\max}(f_n,\mu_n)\frac{\lambda_1^+(f,\mu)}{r}.$$
Moreover,
$$\lim_{n\to+\infty}h_{\mathrm{top}}(f_{n})= h_{\mathrm{top}}(f)+[\frac{d}{2}]\frac{\lambda^+(f)}{r}.$$
(2) For any $d>1$, there exists a compact $d$-dimensional Riemannian manifold $\mathbf M^d$,
for any $+\infty>r>1$, there exists a sequence $\left\{f_{n}\right\}_{n\in\mathbb N^+}$ of $C^r$ diffeomorphisms on $\mathbf M^d$ converging $C^r$ to $f\in\mathrm{Diff}^\infty(\mathbf M^d)$ and a sequence $\left\{\mu_{n}\right\}_{n\in\mathbb N^+}$ of $\left\{f_{n}\right\}_{n\in\mathbb N^+}$-invariant measures converging to an $f$-invariant measure $\mu$, such that
$$\lim_{n\to+\infty}h_{\mu_n}(f_n)-h_{\mu}(f)=\limsup_{n\to+\infty}d^u_{\max}(f_n,\mu_n)\frac{\lambda^+_{1}(f,\mu)}{r}=\frac{1}{r-1}(\lambda_1^+(f,\mu)-\liminf_{n\to+\infty}\lambda_1^+(f_n,\mu_n)).$$
Moreover,
$$\limsup_{n\to+\infty}d^u_{
\max
}(f_n,\mu_n)=1.$$
\end{maintheorem}

\subsection{The outline of the proof}     
\subsubsection{Proof of Theorem \ref{Theorem A}}
\paragraph{\textbf{Bounding entropy along local unstable manifolds.}}
To refine the estimate with respect to dimension, we use the results of \cite{LY85-2} to estimate entropy in a lower-dimensional geometric setting, thereby reducing the geometric dimension of the object to be reparameterized.

\paragraph{\textbf{Yomdin's reparameterization lemma.}}
To refine the estimate with respect to the asymptotic Lipschitz constant, we strengthen Yomdin's reparameterization lemma \cite{Y87}. The key idea is that finer subsets require fewer reparameterizing maps. Moreover, using an algebraic lemma proved by Burguet in \cite{Bu08}, we simplify the proof of the reparameterization lemma.

\paragraph{\textbf{Applying the reparameterization lemma.}}
To apply the reparameterization lemma effectively, one needs to choose admissible times appropriately and stratify the object to be reparameterized accordingly.

\subsubsection{Proof of Theorem \ref{Corollary E}} Motivated by the ideas of \cite{DN05} and \cite{LSW15}, we show that there exists an abundance of diffeomorphisms such that the critical value for the quantitative loss of upper semicontinuity of metric entropy is precisely the one appearing in Corollary \ref{Corollary A}.

\subsubsection{Proof of Theorem \ref{Theorem B}} Motivated by the result of \cite{B14} in the \(2\)-dimensional setting, we generalize it to higher dimensions and show that the critical value for the quantitative loss of upper semicontinuity of metric entropy is precisely the one appearing in Theorem \ref{Theorem A}.

\section{Preliminaries}
\subsection{Metric entropy}
Let $\mu\in\mathcal{M}(X,T)$. Given $\mathcal Q=\left\{A_1,\cdots,A_k\right\}$ a finite measurable partition of $X$, we define the entropy of $\mathcal Q$ by
$$H_{\mu}(\mathcal Q)=-\sum_{1\le i\le k}\mu(A_i)\log{\mu(A_i)}.$$
Metric entropy of $T$ with respect to $\mathcal Q$ is given by
$$h_{\mu}(T,\mathcal Q)=\lim_{n\to+\infty}\frac{1}{n}H_{\mu}(\bigvee_{i=0}^{n-1}T^{-i}\mathcal Q)$$
where
$$\bigvee_{i=0}^{n-1}T^{-i}\mathcal Q:=\left\{A_1\cap\cdots\cap T^{-(n-1)}A_{n}:A_i\in\mathcal Q,\:1\le i\le n\right\}.$$
Metric entropy of $T$ with respect to $\mu$ is given by
$$h_{\mu}(T)=\sup_{\mathcal Q}h_{\mu}(T,\mathcal Q),$$
where $\mathcal Q$ ranges over all finite measurable partitions of $X$.

Let $\mu\in\mathcal{M}(X,T)$. Given $\mathcal Q=\left\{A_1,\cdots,A_k\right\},\:\mathcal P=\left\{B_1,\cdots,B_s\right\}$ two finite measurable partitions of $X$, we define the conditional entropy of $\mathcal P$ with respect to $\mathcal Q$ by
$$H_{\mu}(\mathcal P|\mathcal Q)=\sum_{1\le j\le k}\mu(A_j)H_{\mu|_{A_j}}(\mathcal P),$$
where $\mu|_{A_j}(.):=\frac{\mu(.\cap A_j)}{\mu(A_j)}$.
It is classical to prove that
\begin{equation}
h_{\mu}(T,\mathcal Q)=\lim_{n\to+\infty}H_{\mu}(\mathcal Q|\bigvee_{i=1}^{n}T^{-i}\mathcal Q),
\end{equation}
\begin{equation}
\frac{1}{n}H_{\mu}(\vee_{i=0}^{n-1}T^{-i}\mathcal Q)\downarrow h_{\mu}(T,\mathcal Q),\:\mathrm{as}\:n\to+\infty.
\end{equation}
The following presents several fundamental properties of metric entropy that play a crucial role in estimating the upper bounds of metric entropy.
\begin{lemma}\cite[Lemma 3.2]{B72}\label{Lemma 2.1}
 Let $\left\{\mathcal Q_m\right\}_{m\in\mathbb     N^+}$ be a sequence of finite Borel partitions of $X$ with $\mathrm{diam} \mathcal Q_m\to 0$, as $m\to+\infty$. Then $h_{\mu}(T,\mathcal Q_m)\to h_{\mu}(T)$, as $m\to+\infty$.
 \end{lemma}

\begin{lemma}\label{Lemma 2.2}
Assume that \( \left\{T_n\right\}_{n\in\mathbb  N^+}\) is a sequence of homeomorphisms on $X$ converging uniformly to a homeomorphism \( T \), \( \mu \in \mathcal{M}(X,T) \), and \( \mathcal Q \) is a finite partition of \( X \) with \( \mu(\partial \mathcal Q) = 0 \). If \( \mu_n \in \mathcal{M}(X,T_n) \) converges to \( \mu \) as \( n \to +\infty \), then for every \( m \in \mathbb{N}^+ \),

\[
H_{\mu_n} \Bigl( \bigvee_{i=0}^{m-1} T_n^{-i} \mathcal Q \Bigr) \longrightarrow H_{\mu} \Bigl( \bigvee_{i=0}^{m-1} T^{-i} \mathcal Q \Bigr), \:\mathrm{as}\: n \to +\infty.
\]
\end{lemma} 
\begin{proof}
It suffices to prove that, for any Borel set $A_0,\cdots,A_{m-1}\subset X$ with $\mu(\partial A_j)=0,\:0\le j\le m-1$, one has
$$\lim_{n\to+\infty}\mu_n(A_0\cap T_{n}^{-1}A_1\cap\cdots\cap T_{n}^{-(m-1)}A_{m-1})=\mu(A_0\cap T^{-1}A_1\cap\cdots\cap T^{-(m-1)}A_{m-1}).$$
For $n\in\mathbb N^+$, define $$F_n:X\to X^{m},\:F_n(x)=(x,T_nx,\cdots,T_n^{m-1}x),$$
$$F:X\to X^{m},\:F(x)=(x,Tx,\cdots,T^{m-1}x).$$
and $$\nu_n(.)=(F_n)_{*}\mu_n(.)=\mu_n\circ F_n^{-1}(.),\nu(.)=F_{*}\mu(.)=\mu\circ F^{-1}(.)\in \mathcal M(X^m).$$ For $R=A_{0}\times A_1\times\cdots\times A_{m-1}\in X^{m}$, one has
$$\nu_{n}(R)=\mu_n(A_0\cap T_{n}^{-1}A_1\cap\cdots\cap T_{n}^{-(m-1)}A_{m-1}),\:\forall\:n\in\mathbb N^+,$$
$$\nu(R)=\mu(A_0\cap T^{-1}A_1\cap\cdots\cap T^{-(m-1)}A_{m-1}).$$
Let $\Gamma_m=\left\{Fx:x\in X\right\}$, then one has
$$\nu(\Gamma_m)=1,$$
$$\Gamma_{m}\cap\partial R\subset \cup_{j=0}^{m-1}\left\{(x,Tx,\cdots,T^{m-1}x):T^jx\in\partial A_{j}\right\}.$$
Therefore,
\begin{equation*}
\begin{aligned}
\nu(\partial R)=\nu(\Gamma_m\cap \partial R)&\le \sum_{j=0}^{m-1}\nu(\left\{(x,fx,\cdots,f^{m-1}x):f^jx\in\partial A_{j}\right\})\\
&= \sum_{j=0}^{m-1}\mu(f^{-j}\partial A_{j})=0.
\end{aligned}
\end{equation*}
It suffices to prove that $\nu_n\to\nu$, as $n\to+\infty$. Choose $\Phi\in C(X^m)$, for $n\in\mathbb N^+$, let $\psi_n(x)=\Phi(F_n(x)),\:\psi(x)=\Phi(F(x))$. Then, one has
\begin{equation*}
\begin{aligned}
\left|\int\Phi d\nu_n-\int \Phi d\nu\right|&=\left|\int\psi_n d\mu_n-\int \psi d\mu\right|\\
&\le \left|\int\psi_n d\mu_n-\int \psi d\mu_n\right|+\left|\int\psi d\mu_n-\int \psi d\mu\right|\\
&\le \sup_{x\in X}\left|\psi_n(x)-\psi(x)\right|+\left|\int\psi d\mu_n-\int \psi d\mu\right|\to 0,\:\mathrm{as}\:n\to+\infty.
\end{aligned}    
\end{equation*}

\end{proof}
\begin{lemma}\cite[Lemma 8]{B23}\label{Lemma 2.3}
Let $\mu\in\mathcal{M}(X)$ and $E$ be a finite subset of $\mathbb N$. For any finite partition $\mathcal Q$ of $X$, for any $m\in\mathbb{N^+}$, we have with $\mu^E(.):=\frac{1}{\#E}\sum_{j\in E}T^j_{*}\mu(.)=\frac{1}{\# E}\sum_{j\in E}\mu\circ T^{-j}(.)$ and $\mathcal Q^E:=\vee_{j\in E}T^{-j}\mathcal Q,\:\mathcal Q^m:=\vee_{j=0}^{m-1}T^{-j}\mathcal Q$, one has
$$\frac{1}{\#E}H_{\mu}(\mathcal Q^E)\le\frac{1}{m}H_{\mu^E}(\mathcal Q^m)+6m\frac{\#(E+1)\triangle E}{\#E}\log\#\mathcal Q.$$
\end{lemma}
\subsection{Topological entropy}
The following presents the basic definition and fundamental properties of topological entropy.
\begin{definition}
Let $\xi$ be an open cover of $X$, and let $N_n(T,\xi)$ be the minimal cardinality of a subcover of $$\mathcal \xi^n:=\vee_{i=0}^{n-1}T^{-i}\mathcal \xi=\left\{A_1\cap\cdots\cap T^{-(n-1)}A_{n}:A_i\in \xi,1\le i\le n\right\}.$$ 
We define topological entropy of $T$ with respect to $\xi$ by
$$h_{\mathrm{top}}(T,\xi)=\lim_{n\to+\infty}\frac{1}{n}\log N_n(T,\xi),$$
and define topological entropy of $T$ by 
$$h_{\mathrm{top}}(T)=\sup_{\xi}h_{\mathrm{top}}(T,\xi),$$
where $\xi$ ranges over all open covers of $X$.
\end{definition}
The variational principle,
$$h_{\mathrm{top}}(T)=\sup_{\mu\in\mathcal M(X,T)}h_{\mu}(T)=\sup_{\mu\in\mathcal M^{erg}(X,T)}h_{\mu}(T),$$
establishes the relationship between topological entropy and metric entropy.

\subsection{The unstable manifolds}
For $f\in\mathrm{Diff}^r(\mathbf M^d)\:(r>1)$ and $x\in\Gamma_f$, 
if $\lambda_i(f,x)>0$, define
$$W^{i,u}(f,x)=\left\{y\in \mathbf M^d:\limsup_{n\to+\infty}\frac{1}{n}\log{\mathrm d(f^{-n}x,f^{-n}y)}\le -\lambda_i(f,x)\right\}.$$
By \cite{PY76}, $W^{i,u}(f,x)$ is a $C^r$ $\mathrm{dim}(E^1_{f}(x)\oplus\cdots\oplus E^i_{f}(x))$-dimensional immersed sub-manifold of $\mathbf M^d$ tangent at $x$ to $E^1_{f}(x)\oplus\cdots\oplus E^i_{f}(x)$. It is called the $i^{th}$ unstable manifold of $f$ at $x$. We sometimes refer to 
$$\left\{W^{i,u}(f,x):x\in\Gamma_{f}\right\}$$ 
as the $W^{i,u}(f,.)$-foliation on $\mathbf M^d$. For $x\in\Gamma_f$, as a $C^r$ immersed sub-manifold, $W^{i,u}(f,x)$ inherits a Riemannian structure from $\mathbf M^d$, which gives rise to a Riemannian metric on each leaf of $W^{i,u}(f,.)$. We denote the metric by $\mathrm d^{i,u}$. The measure $\nu\in\mathcal M(\mathbf M^d,f)$ with $\lambda_i(f,x)>0,\:\nu$-a.e.\:$x\in\mathbf M^d$ defines conditional measure on the leaves of $W^{i,u}(f,.)$. More precisely, a measurable partition $\zeta$ of $\mathbf M^d$ is said to be subordinate to the $W^{i,u}(f,.)$-foliation if 
$$\mathrm{for}\:\nu\mathrm{-a.e.}\:x\in\mathbf M^d,\:\zeta(x)\subset W^{i,u}(f,x),\:\mathrm{and}\:\zeta(x)\:\mathrm{contains\:a\:neighborhood\:of}\: W^{i,u}(f,x).$$ Associated with each measurable partition subordinate to $W^{i,u}(f,.)$ is a system of conditional measures
$$\nu=\int_{\Gamma_{f}}\nu^{i}_{x}d\nu(x),$$
where $\nu_x^{i}$ is the conditional measure  with respect to $\zeta(x)\subset W^{i,u}(f,x)$.

Let $\varepsilon>0$. For $x\in\Gamma_f$ and $n\in\mathbb{N}^+$, define
$$V^{i,u}_{f}(x,n,\varepsilon)=\left\{y\in W^{i,u}(f,x):\mathrm d^{i,u}(f^kx,f^ky)\le\varepsilon,\:0\le k<n\right\}.$$
Define
$$\underline{h}^{i}_{\nu}(f,x,\varepsilon,\zeta)=\liminf_{n\to+\infty}-\frac{1}{n}\log{\nu_x^{i}(V_{f}^{i,u}(x,n,\varepsilon)}),$$
$$\overline{h}^{i}_{\nu}(f,x,\varepsilon,\zeta)=\limsup_{n\to+\infty}-\frac{1}{n}\log{\nu_x^i(V^{i,u}_{f}(x,n,\varepsilon))}.$$
The following lemma establishes the relationship between the exponential growth rate of Bowen balls on unstable manifolds with respect to conditional measures and metric entropy.
\begin{lemma}\cite[Proposition 7.2.1, Corollary 7.2.2]{LY85-2}\label{Lemma 2.4}
Assume that $\nu\in\mathcal{M}(\mathbf M^d,f)$ with $\nu$-a.e.\:$x\in\mathbf M^d$, $\lambda_1(f,x)>0$. Then for $\nu$-$\mathrm{a.e.\:}x\in\mathbf M^d$ and $1\le i\le u(f,x)$, one has
$$h^{i}_{\nu}(f,x,\zeta)=\lim_{\varepsilon\to 0}\underline{h}^{i}_{\nu}(f,x,\varepsilon,\zeta)=\lim_{\varepsilon\to 0}\overline{h}^{i}_{\nu}(f,x,\varepsilon,\zeta),\:h_{\nu}(f)=\int h^{u(f,x)}_{\nu}(f,x,\zeta)d\nu(x),$$
where 
$$u(f,x)=\max\left\{i\in\mathbb N^+:\lambda_i(f,x)>0\right\},$$ 
for $x\in\mathbf M^d$. Moreover, if we assume that $\nu\in\mathcal{M}^{erg}(\mathbf M^d,f)$ with $\lambda_1(f,\nu)>0$, then there exists $u\in\mathbb N^+$, such that for $\nu$-a.e.\:$x\in\mathbf M^d$, one has
$$u(f,x)=u,\:h_{\nu}(f)=h^u_{\nu}(f,x,\zeta).$$
\end{lemma}
For $x\in\Gamma_f$, if $i=u(f,x)$, we denote $W^{u}(f,x):=W^{u(f,x),u}(f,x),\:\nu_x:=\nu_{x}^{u(f,x)}$\:(if there exist). Moreover, let $W^u_{loc}(f,x)$ denote the local unstable manifold at $x$ with a sufficiently small size. 
\begin{lemma}\label{Lemma 2.5}
Let $f\in\mathrm{Diff}^r(\mathbf M^d)\:(r>1)$ and $\nu\in\mathcal M^{erg}(\mathbf M^d,f)$ satisfy $\lambda_1(f,\nu)>0$. For any $\varepsilon>0$, there exists a compact subset $E$ of $\mathbf M^d$ with
$\nu(E)>1-\varepsilon$, such that for any $\tau>0$, there exists $\rho>0$, for any $x\in E$, for any measurable set $\Sigma\subset W^u_{loc}(f,x)$ with $\nu_{x}(\Sigma\cap E)>0$, and for any finite partition $\mathcal P$ with $\mathrm{diam}\mathcal P<\rho$, we have
\begin{equation}\label{3.2'}
h_{\nu}(f)\le \liminf_{n\to+\infty}\frac{1}{n}H_{\nu_{x,E,\Sigma}}(\mathcal P^n)+\tau,
\end{equation}
where $\nu_{x,E,\Sigma}(.)=\frac{\nu_x(.\:\cap E\cap \Sigma)}{\nu_x(E\cap \Sigma)}$.  \end{lemma}
\begin{proof}
By Lemma \ref{Lemma 2.4} and the Egorov's theorem, for any $\varepsilon>0$,
there exists a compact subset $E$ of $\mathbf M^d$ with $\nu(E)>1-\varepsilon$, for any $\tau>0$, there exists $\rho>0$ such that
\begin{equation}\label{2.4'}
\forall\:x\in E,\:\liminf_{n\to+\infty}-\frac{1}{n}\nu_{x}(B_n(x,\rho))\ge h_{\nu}(f)-\tau,
\end{equation}
where $B_n(x,\rho)=\left\{y\in\mathbf M^d:\mathrm{d}(f^ix,f^iy)\le\rho,\:0\le i<n\right\}$. Therefore, by \eqref{2.4'}, we have
\begin{equation}
\begin{aligned}
\liminf_{n\to+\infty}\frac{1}{n}H_{\nu_{x,E,\Sigma}}(\mathcal P^n)&=\liminf_{n\to+\infty}\int-\frac{1}{n}\log\nu_{x,E,\Sigma}(\mathcal P^n(y))d\nu_{x,E,\Sigma}(y)\\
&\mathbf{\text{Let $\mathcal P^n(y)$ be the element of $\mathcal P^n$ containing $y$}}\\
&\ge\int\liminf_{n\to+\infty}-\frac{1}{n}\log\nu_{x,E,\Sigma}(\mathcal P^n(y))d\nu_{x,E,\Sigma}(y)\\
&\ge \int\liminf_{n\to+\infty}-\frac{1}{n}\log\nu_{x}(\mathcal P^n(y))d\nu_{x,E,\Sigma}(y)\\
&\ge \int\liminf_{n\to+\infty}-\frac{1}{n}\log\nu_{y}(\mathcal P^n(y))d\nu_{x,E,\Sigma}(y)\\
&\nu_{x,E,\Sigma}\mathrm{-a.e.}\:y\in\mathbf M^d,\:\nu_{y}=\nu_{x}\\
&\ge\int\liminf_{n\to+\infty}-\frac{1}{n}\log\nu_{y}(B_{n}(y,\rho))d\nu_{x,E,\Sigma}(y)\\
&\mathrm{diam}\mathcal P<\rho\\
&\ge h_{\nu}(f)-\tau\\
&\nu_{x,E,\Sigma}\text{-a.e.}\:y\in\mathbf M^d,\:\mathrm{one\:has}\:y\in E.
\end{aligned}    
\end{equation}
\end{proof}
\subsection{\texorpdfstring{Review of the $C^r$ size of maps}{Review of the Cr size of maps}}
Assume that $X$ is a compact metric space. For a continuous map $F:X\to\mathbb{R}^{d}$, denote
$$\left\|F\right\|_{0}=\max_{x\in X}\left\|F(x)\right\|.$$

Let $U\subset \mathbb R^m$ be an open concave set. Given $r\in\mathbb N$, we say that a map $F:U\to\mathbb R^d$ is $C^r$ if for any $\omega\in \mathbb N^m$ with $1\le\left|\omega\right|:=\omega_1+\cdots+\omega_{m}\le r$, one has 
$$\partial ^{\omega}F:=\frac{\partial^{\omega_1+\cdots+\omega_{m}}F}{\partial^{\omega_1}x_1\cdots\partial^{\omega_{m}}x_{m}}$$
exists and is continuous on $U$. For any compact subset $K\subset U$, we define the $C^r$ norm
$$\left\|F\right\|_{r,K}:=\max_{1\le \left|\omega\right|\le r}\max_{x\in K}\left\|\partial^\omega_{x} F\right\|.$$
Given $\alpha\in (0,1)$, we say a map $F$ is $C^{\alpha}$ if for any compact set $K\subset U$,
$$\left\|F\right\|_{\alpha,K}:=\sup_{x\ne y\in K}\frac{\left\|F(x)-F(y)\right\|}{\left\|x-y\right\|^{\alpha}}<+\infty.$$
Given $r=[r]+\alpha>1$ which is not an integer, where $0<\alpha<1$, we say $F$ is $C^r$, if it is $C^{[r]}$ and each derivative $\partial^{\omega}F$ is $C^{\alpha}$, for all  $\omega\in \mathbb N^m$ with $\left|\omega\right|=[r]$. For any compact subset $K\subset U$, we define the $C^r$ norm
$$\left\|F\right\|_{r,K}:=\left\|F\right\|_{[r],K}+\max_{\left|\omega\right|=[r]}\left\|\partial^{\omega}F\right\|_{\alpha,K}.$$
If $F$ is a $C^r$ diffeomorphism,
 we define the $C^r$ norm
$$\left\|F\right\|_{r,K}:=\max\left\{\left\|F\right\|_{[r],K}+\max_{\left|\omega\right|=[r]}\left\|\partial^{\omega}F\right\|_{\alpha,K},\left\|F^{-1}\right\|_{[r],K}+\max_{\left|\omega\right|=[r]}\left\|\partial^{\omega}F^{-1}\right\|_{\alpha,K}\right\}.$$
Let $\Omega$ be a compact subset of $\mathbb R^m$ which is equal to the closure of its interior. A map $F:\Omega\to \mathbb R^d$ is $C^r$, if $F$ has a $C^r$ extension to an open neighborhood of $\Omega$. In this case,
$$\left\|F\right\|_{r}=\sup_{K\subset \mathrm{int}(\Omega)}\left\|F\right\|_{r,K}.$$

The definition of $C^r$ maps on subsets of Euclidean space can be naturally extended, via local coordinate charts, to maps between compact Riemannian manifolds. A $C^r$ structure on a smooth manifold $N$ is defined by a maximal atlas $\mathcal B$ with $C^r$ changes of coordinates. A smooth manifold equipped with a $C^r$ structure $\mathcal B$ is called a $C^r$ manifold. A finite subset of $\mathcal B$ that covers $N$ is called a $C^r$ atlas of $N$. Let $N_1$, $N_2$ be two compact $C^r$ manifolds without boundary, and let $\mathcal B_i$ be finite $C^r$ atlases of $N_i$, for $i=1,2$. We say that $F:N_1\to N_2$ is $C^r$, if each map $\tau_2^{-1}\circ F\circ\tau_1$, where $\tau_i$ ranges over $\mathcal B_i$, for $i=1,2$, is $C^r$. The norm of $f$ is:
$$\left\|F\right\|_{r}=\max_{\tau_i\in \mathcal B_i,i=1,2}\left\|\tau_2^{-1}\circ F\circ \tau_1\right\|_{r}<+\infty.$$
The $C^r$ norm is independent of the choice of local coordinate charts, up to the equivalence of the norms.

Similarly, if $F:N_1\to N_2$ is continuous, denote
$$\left\|F\right\|_{0}=\max_{\tau_2\in\mathcal B_2}\left\|\tau_{2}^{-1}\circ F\right\|_{0}.$$

The following presents several fundamental properties of derivatives that will be utilized in this paper.

For positive integers $m,p,q$, let $M_{p,q}(\mathbb R)$ be the set of all real valued $p\times q$ matrices and  denote $AB\in M_{p,m}(\mathbb R)$ the product of two matrices $A\in M_{p,q}(\mathbb R)$, $B\in M_{q,m}(\mathbb R)$. We have with the standard multi-index notations:\\
$\bullet\:$General Leibniz rule: Let $u:\mathbb R^d\to M_{p,q}(\mathbb R)$ and $v:\mathbb R^d\to M_{q,m}(\mathbb R)$, be $C^r$ maps, then for any $\alpha=(\alpha_1,\cdots,\alpha_{d})\in\mathbb N^d$ with $\left|\alpha\right|\le r,$ we have
$$\partial^{\alpha}(uv)=\sum_{\beta\le\alpha}\binom{\alpha}{\beta}(\partial^{\beta}u)(\partial^{\alpha-\beta}v).$$
$\bullet\:$Faa di Bruno's formula: Let $u:\mathbb R^d\to\mathbb R^{c}$ and $v:\mathbb R^{e}\to\mathbb R^d$ be $C^r$ maps, then for any $\alpha\in \mathbb N^e$ with $\left|\alpha\right|\le r$, for any $1\le j\le c$, we have
$$\partial ^{\alpha}(u_{j}\circ v)=\sum_{\left|\beta\right|\le \left|\alpha\right|,\beta\in\mathbb N^d}(\partial^{\beta}u_{j})\circ v\times P_{\beta}((\partial^{\gamma}v_i)_{\gamma,i}),$$
where $P_{\beta}((\partial^{\gamma}v_i)_{\gamma,i})$ is a universal polynomial, in $\partial^{\gamma}v_i$ for $i=1,\cdots,d$ and $\gamma\in\mathbb N^e$ with $\left|\gamma\right|\le \left|\alpha\right|$, of total degree less than or equal to $\left|\alpha\right|$.
\subsection{Taylor's expansion}
Assume that $\phi:U\to\mathbb R^d$ is $C^{r}$, where $U\subset \mathbb R^{m}$ is an open concave set, $r=[r]+\alpha$, and $d,m\in\mathbb N^+$. We consider the following Taylor expansion at $x\in U$ at the level $[r]$,
$$\phi(x+a)=\sum_{k=0}^{[r]}\frac{1}{k!}[D^k_x\phi](a)^k+R_{[r]}(x,a),$$
where $a\in\mathbb R^{m}$ with $x+a\in U$, $(a)^k=(a,\dots,a)\in(\mathbb R^{m})^k$, and
$$R_{[r]}(x,a)=\frac{1}{([r]-1)!}\int_{0}^1(1-t)^{[r]-1}([D^{[r]}_{(x+ta)}\phi-D^{[r]}_{x}\phi](a)^{[r]})dt.$$
With the Hölder condition, one has 
$$\left\|R_{[r]}(x,a)\right\|_{0}\le\frac{1}{[r]!}\left\|D^{[r]}\phi\right\|_{\alpha}\left\|a\right\|^{r}.$$
For $1\le k\le [r]$,
consider the Taylor expansion of $D^{k}\phi$ at $x\in U$ at the level $[r]-k$, one has
$$\left\|D^k R_{[r]}(x,a)\right\|_{0}\le \frac{1}{[r-k]!}\left\|D^{[r]}\phi\right\|_{\alpha}\left\|a\right\|^{r-k}.$$

\subsection{The reparametrization lemma}
\subsubsection{Semi-algebraic sets and the algebraic Lemma}
\begin{definition}
A semi-algebraic set is a subset of $\mathbb R^d$ that can be described by a finite number of polynomial equations and inequalities. More precisely: A set $S\subset \mathbb R^d$ is called a semi-algebraic set if it can be expressed as a finite  Boolean combination (using unions and intersections) of sets of the form
$$\left\{s\in\mathbb R^d:g_1(s)=0\right\},$$
$$\left\{s\in\mathbb R^d:g_2(s)> 0\right\},$$
where
$$g_1,g_2\in \mathbb R[x_1,\cdots,x_d]$$
are polynomials in $n$ variables. Equivalently, $S$ can be written as
$$S=\cup_{1\le i\le s_1}\cap_{1\le j\le s_2}\left\{s\in\mathbb R^d:g_{i,j}(s)*_{i,j}0\right\},$$
where $s_1,s_2\in\mathbb N$, each $g_{i,j}\in\mathbb R[x_1,\cdots,x_d]$, 
 and each ``$*_{i,j}$" is ``$=$" or ``$>$". 
\end{definition}
We present the algebraic lemma in the formulation stated in \cite{OB25}, which refers to \cite{Bu08}.

\begin{lemma}\cite[Lemma 4.2]{OB25}
Let $P:[0,1]^{k}\to\mathbb R^d$ be a polynomial map with total degree less than or equal to $r$ and let $Y$ be a bounded semi-algebraic set of $\mathbb R^d$. Then there is a constant $B_{r,d}$ depending only $r,d,\mathrm{deg}Y,\mathrm{diam}Y$ and semi-algebraic analytic injective maps $\theta_i:(0,1)^{k_i}\to [0,1]^{k}$, $k_{i}\le k$, $i\in I$, such that\\
(1)\:$\#I\le B_{r,d}$,\\
(2)\:$\left\|\theta_{i}\right\|_{r+1}\le\frac{1}{100d}$, $\left\|P\circ\theta_{i}\right\|_{r+1}\le \frac{1}{100d}$,\\
(3)\:$\cup_{i\in I}\mathrm{Im}\theta_i=P^{-1}[Y]$.
\end{lemma}
\begin{remark}
(1) Maps $\theta_i$ may be $C^{r}$ extended on $[0,1]^{k_i}$ as $\theta_i$ satisfies $\left\|\theta_i\right\|_{r+1}\le 1$.\\
(2) By the invariance of domain theorem the image of each map $\theta_i$ is open and each $\theta_i$ is a homeomorphism onto its image.
\end{remark}
\subsubsection{Yomdin reparametrization lemma}
For $g\in\mathrm{Diff}^r(\mathbf M^d)$ and $x\in\mathbf M^d$, let $k_{g}(x)=\log{\left\|D_xg\right\|}$ and $k_{g}^+(x)=\log^+{\left\|D_xg\right\|}$.
\begin{mainlemma}\label{2.6}
For any $\gamma>0$, there exists \(\varepsilon_{\gamma} > 0\), for any $g\in\mathrm{Diff}^r(\mathbf M^d)\:(r>1)$ with $\left\|g\right\|_{r}<\gamma$, for any \(0 < \varepsilon < \varepsilon_\gamma\), for any ball $B(y,\varepsilon)\subset \mathbf M^d$, and any $C^r$ map \(\sigma:[0,1]^m\to\mathbf M^d\) with $\left\|\sigma\right\|_{r}\le \varepsilon$, there exists a constant \(C_{r,m,d} > 0\) (depending only on \(r,m,d\)), for any $\mathbf k\ge 0$, there exists a family of $C^r$ maps \(\Theta\), such that:\\
1. for any \(\psi \in \Theta\), \(\psi: [0,1]^{k'_{\psi}}\to [0,1]^{m}\) with $1\le k'_{\psi}\le m$;\\
2.\:$\sigma^{-1}\left\{x \in \mathbf M^d : k_g^+(x) = \mathbf k\right\}\cap (g\circ\sigma)^{-1}B(y,\varepsilon) \subset \bigcup_{\psi \in \Theta} \psi([0,1]^{k'_{\psi}});$
   \\  
3.\:for any \(\psi \in \Theta\), $\left\|g\circ\sigma\circ\psi\right\|_{r}\le \varepsilon$;\\  
4.\:$
   \#\Theta \leq C_{r,m,d} e^{\frac{m\mathbf k}{r}};$\\
5. for any $\psi\in\Theta$, $\left\|D\psi\right\|_{0}\le 1$.  
\end{mainlemma}

\begin{proof}
\textbf{First step: a simplification of the proof.} In this proof, we denote $a\lesssim b$, if there is a constant $C_{r,m,d}$, such that $a\le C_{r,m,d}b$. We denote $a\approx b$, if $a\lesssim b$ and $b\lesssim a$.  For any \(\mathbf k\ge 0\), it suffices to verify that there exists a family of $C^r$ maps \(\Theta\) (where each \(\psi \in \Theta\), \(\psi: [0,1]^{k'_{\psi}}\to [0,1]^{m}\) with $1\le k'_{\psi}\le m$), satisfying: \\ 
1.\: $
   \sigma^{-1}\left\{x \in \mathbf M^d : k_g^+(x) = \mathbf k\right\}\cap (g\circ\sigma)^{-1}B(y,\varepsilon)  \subset \bigcup_{\psi \in \Theta} \psi([0,1]^{k'_{\psi}});$ \\ 
2.\:for any \(\psi \in \Theta\), $\left\|g\circ\sigma\circ\psi\right\|_{r}\lesssim\varepsilon;$ \\ 
3.\:$
   \#\Theta \lesssim  e^{\frac{m}{r}\mathbf k};$ \\
4.\:for any $\psi\in\Theta$, $\left\|D\psi\right\|_{0}\le 1$.\\
Let $\alpha=r-[r]$. We choose $\varepsilon_{\gamma}>0$ small enough, such that for every $0<\varepsilon<\varepsilon_{\gamma}$, for every $C^r$ diffeomorphism $g$ satisfying
$$\left\|g\right\|_{r}<\gamma,$$
we have
$$\max_{s=1,\cdots,[r]}\left\|D^{s}g^x_{m\varepsilon}\right\|_{0}\le 2m\varepsilon\left\|D_{x}g\right\|,\:\:\mathrm{and}\:\left\|D^{[r]}g^{x}_{m\varepsilon}\right\|_{\alpha}\le 2m\varepsilon\left\|D_xg\right\|,$$ for all $x\in\mathbf M^d$, where $$g^{x}_{m\varepsilon}:\left\{\omega\in T_{x}\mathbf M^d:\left\|\omega\right\|\le1\right\}\to\mathbf M^d,\omega\mapsto g\circ\mathrm{exp}_{x}(m\varepsilon\omega).$$ 
Let $R_{\mathrm{inj}}>0$ be the injectivity radius of $\mathbf M^d$. By scaling the Riemannian metric by a constant factor, the injectivity radius $R_{\mathrm{inj}}$ can be normalized to be greater than 
$m$. Since we do not care about the constant $C_{r,m,d}$, this normalization does not affect the conclusions. 
Choose $z\in\mathbf M^d$, such that there exists $s\in [0,1]^m$,\\
$\bullet\:z=\sigma(s)$,\\
$\bullet\:k_{g}^+(z)=\mathbf k$,\\
$\bullet\:g(z)\in B(y,\varepsilon)$.\\
Without loss of generality,
$$\mathrm{Im}\sigma\subset\mathrm{exp}_{z}\left\{\omega\in T_{z}\mathbf M^d:\left\|\omega\right\|\le m\varepsilon\right\}.$$
We assume that $\varepsilon=\frac{1}{m}$ through the local charts 
$$g\mapsto (m\varepsilon)^{-1}\exp_{g(z)}^{-1}\circ g\circ\exp_{z}(m\varepsilon(.)):\left\{\omega\in T_{z}\mathbf M^d:\left\|\omega\right\|\le 1\right\}\to T_{g(z)}\mathbf M^d,$$ 
$$\sigma\mapsto (m\varepsilon)^{-1}\exp^{-1}_{z}\circ\sigma:[0,1]^m\to \left\{\omega\in T_z\mathbf M^d:\left\|\omega\right\|\le 1\right\},$$
and 
$$B(y,\varepsilon)\mapsto (m\varepsilon)^{-1}\exp_{g(z)}^{-1}B(y,\varepsilon)\underset{\approx}{\subset}  B_{T_{g(z)}\mathbf M^d}(0,\frac{2}{m})$$
with
$$\mathrm{exp}_{g(z)}(B_{T_{g(z)}\mathbf M^d}(0,\frac{2}{m}))\underset{\approx}{\subset} B(y,\frac{4}{m}).$$

\textbf{Second step: Taylor polynomial approximation.} One computes for an affine map 
$$\psi_1:[0,1]^{m}\to [0,1]^{m},(x_1,\cdots,x_m)\mapsto (b^{\frac{1}{m}}x_1+c_1,\cdots,b^{\frac{1}{m}}x_m+c_m),$$ 
with $0<b^{\frac{1}{m}}\le1$ and $0\le c_i\le 1-b^{\frac{1}{m}}\:(1\le i\le m)$ precised later. 
Then, by the Faa di Bruno's formula and the general Leibiniz rule, we have
\begin{equation*}
\begin{aligned}
\left\|D^{[r]}(g\circ\sigma\circ\psi_1)\right\|_{\alpha}&\lesssim b^{\frac{[r]+\alpha}{m}}\left\|D^{[r]}(g^z\circ\sigma^{z})\right\|_{\alpha}\\
&\text{with\:$\sigma^z:=\mathrm{exp}_{z}^{-1}\circ\sigma$ }\\
&\lesssim b^{\frac{r}{m}}\left\|D^{[r]-1}(D_{\sigma^z}g^z\circ D_{}\sigma^{z})\right\|_{\alpha}\\
& \lesssim b^{\frac{r}{m}}\max_{s=0,\cdots,[r]-1}\left\{\left\|D^{s}(D_{\sigma^z}g^z)\right\|_{0},\left\|D^{[r]-1}(D_{\sigma^z}g^z)\right\|_{\alpha}\right\}\\
&\max_{s=1,\cdots,[r]}\left\{\left\|D^{s}\sigma^z\right\|_{0},\left\|D^{[r]}\sigma^z\right\|_{\alpha}\right\}.
\end{aligned}    
\end{equation*}
 Moreover, 
 $$\max_{s=1,\cdots,[r]}\left\{\left\|D^{s}\sigma^z\right\|_{0},\left\|D^{[r]}\sigma^z\right\|_{\alpha}\right\}\lesssim 1,$$ 
 as $\left\|\sigma\right\|_{r}\le 1$. Therefore, by the Faa di Bruno's formula and the choice of $\varepsilon=\frac{1}{m}$, we have
$$\max_{s=0,\cdots,[r]-1}\left\{\left\|D^{s}(D_{\sigma^z}g^z)\right\|_{0},\left\|D^{[r]-1}(D_{\sigma^z}g^z)\right\|_{\alpha}\right\}\lesssim \left\|D_{z}g\right\|.$$
Above all, we have
\begin{equation*}
\begin{aligned}
\left\|D^{[r]}(g\circ\sigma\circ\psi_1)\right\|_{\alpha}&\lesssim b^{\frac{r}{m}}\left\|D_{z}g\right\|\\ 
&\mathrm{with}\:b\approx e^{-\frac{m}{r}\mathbf k}\\
& \le \frac{1}{2d}.
\end{aligned}    
\end{equation*}
It is clear that there exists a family of affine maps $\Theta_1$, such that\\

$\bullet\:
   \sigma^{-1}\left\{x \in \mathbf M^d : k_g^+(x) = \mathbf k\right\}\cap (g\circ\sigma)^{-1}B(y,\frac{4}{m})  \subset \bigcup_{\psi_1\in\Theta_1} \mathrm{Im}\psi_1;\\$

 $\bullet\:\left\|D^{[r]}(g\circ\sigma\circ\psi_1)\right\|_{\alpha}\le \frac{1}{2d},$ for all $\psi_1\in\Theta_1$,\\

 $\bullet\:\#\Theta_1\lesssim e^{\frac{m}{r}\mathbf k}.$\\
  
Therefore, for any $\psi_1\in\Theta_1$, the Taylor polynomial $P$ at $0$ of degree $[r]$ of $g\circ\sigma\circ\psi_1$ satisfies:
$$\left\|P-g\circ\sigma\circ\psi_1\right\|_{r}\le \frac{1}{2d},\:\left\|P-g\circ\sigma\circ\psi_1\right\|_{0}\le \frac{1}{2d}.$$ 

\textbf{Third step: estimation of the cardinality of semi-algebraic sets via the coverage of cubes.} We have
$$ (g\circ\sigma\circ\psi_1)^{-1} B(y,\frac{4}{m})\subset P^{-1}B(y,\frac{8}{m}).$$
we apply now the algebra lemma to $P$ with $Y$ being the ball $B(y,\frac{8}{m})$. Let $\Theta_{P}$ be the family of reparametrization obtained in this way. For $\psi_{2}\in\Theta_{P}$, note that $$\left\|P\circ\psi_{2}\right\|_{r}\le \frac{1}{100d},\:\left\|\psi_{2}\right\|_{r}\le\frac{1}{100d}.$$ Therefore, we obtain, $$\left\|g\circ\sigma\circ\psi_1\circ\psi_2\right\|_{r}\le \left\|(g\circ\sigma\circ\psi_1-P)\circ\psi_2\right\|_{r}+\left\|P\circ\psi_2\right\|_{r} \le 1.$$ 
Then,
$$\#\Theta\le\#\Theta_1\times\#\Theta_{P}\lesssim e^{\frac{m}{r}\mathbf k} .$$
By
$$(g\circ\sigma\circ\psi_1)^{-1}B(y,\frac{4}{m})\subset \cup_{\psi_2\in\Theta_P}\psi_{2}([0,1]^{k_{\psi_2}'}),$$
we obtain
\begin{equation}
\begin{aligned}
\sigma^{-1}\left\{x\in\mathbf M^d:k_{g}^+(x)=\mathbf k\right\}\cap(g\circ\sigma)^{-1}B(y,\frac{4}{m})&\subset \cup_{\psi_1\in\Theta_1}\psi_1([0,1]^m)\cap(g\circ\sigma)^{-1}B(y,\frac{4}{m})\\
&\subset\cup_{\psi_1\in\Theta_1}\psi_1((g\circ\sigma\circ\psi_1)^{-1}B(y,\frac{4}{m}))\\
&\subset \cup_{\psi_1\in\Theta_1}\cup_{\psi_2\in\Theta_P}\psi_1\circ\psi_{2}([0,1]^{k_{\psi_2}'}).
\end{aligned}
\end{equation}

\end{proof}

\subsubsection{Burguet's reparametrization lemma}
Assume that $r=[r]+\alpha>1$.
\begin{definition}{\cite{B23}}
(1) A $C^r$ embedded curve $\sigma:[0,1]\to\mathbf M^d$ is said to be $C^r$ bounded if 
$$\max_{2\le k\le[r]}\left\|D^k \sigma\right\|_{0}\le\frac{1}{6}\left\|D \sigma\right\|_{0},\:\left\|D^{[r]}\sigma\right\|_{\alpha}\le\frac{1}{6}\left\|D \sigma\right\|_{0}.$$
(2) For $\varepsilon>0$, a $C^r$ embedded curve $\sigma$ is said to be strongly $\varepsilon$-bounded if $\sigma$ is $C^r$ bounded and $\left\|D\sigma\right\|_{0}\le \varepsilon$.
\end{definition}
For an embedded curve $\sigma:[0,1]\to\mathbf M^d$ and 
$x\in\mathrm{Im}\sigma$, let $k_{g,\sigma}'(x)=\log\left\|D_xg|_{T_x(\mathrm{Im}\sigma)}\right\|$.
\begin{mainlemma}\label{Lemma B}\cite[Lemma 12]{B23}
For any $\gamma>0$, there exists $\varepsilon_{\gamma}'>0$, for any $g\in\mathrm{Diff}^r(\mathbf M^d)\:(r>1)$ with $\left\|g\right\|_{r}<\gamma$, for any $0<\varepsilon<\varepsilon_{\gamma}'$, and any strongly \(\varepsilon\)-bounded curve $\sigma:[0,1]\to\mathbf M^d$, there exists a constant \(C_{r,d} > 0\) (depending only on \(r,d\)) such that for any \(\mathbf k,\mathbf k'\in \mathbb{R}\), there exists a family of affine maps \(\Theta\) (where each \(\theta\in \Theta\), \(\theta: [0,1]\to [0,1]\)), satisfying:\\ 
1.\:$\sigma^{-1}\left(\left\{x \in \mathbf M^d : k_g(x) = \mathbf k,\:k'_{g,\sigma}(x)=\mathbf k'\right\}\right) \subset \bigcup_{\theta \in \Theta} \theta([0,1]);$\\  
2.\:for any \(\theta\in \Theta\), $g\circ\sigma\circ\theta$ is \(C^r\) bounded;\\  
3.\:$\#\Theta \leq C_{r,d} e^{\frac{\mathbf k-\mathbf k'}{r - 1}}$;\\
4.\:for any $\theta\in\Theta$, $\left\|D\theta\right\|_{0}\le 1$. \end{mainlemma}
\begin{mainlemma}\label{Lemma C}\cite[Lemma 13]{B23}
For any $\varepsilon>0$, for any $C^r$ bounded curve $\sigma$, for any ball $B(y,\varepsilon)\subset \mathbf M^d$, there exists an affine map $\:\theta:[0,1]\to[0,1]$, such that:\\
$\bullet\:\left\|D(\sigma\circ\theta)\right\|_{0}\le 3\varepsilon$;\\
$\bullet\:\sigma^{-1}B(y,\varepsilon)\subset \theta([0,1])$;\\
$\bullet\:\left\|D\theta\right\|_{0}\le 1$.
\end{mainlemma}
\section{Bounding the entropy for the case of ergodic measures}
In this section, we aim to prove the following two propositions, which give an upper bound estimation of metric entropy in the case of ergodic measures.
\begin{mainproposition}\label{Proposition A}
Assume that $g\in\mathrm{Diff}^r(\mathbf M^d)\:(r>1)$, for any $q\in\mathbb N^+$, there exists a $C^r$ neighborhood $\mathcal V_g$ of $g$ and $\varepsilon_{q}(g)>0$, for any $f\in\mathcal V_g$, for any $\mu\in\mathcal M^{erg}(\mathbf M^d,f)$ with $\mathrm{dim}E_f^u(z)=d_{u}$,
for $\mu$-a.e.\:$z\in\mathbf M^d$, there exists a constant $C_{r,d_u,d}>0$, for any finite partition $\mathcal Q$ of $\mathbf M^d$ with $\mathrm{diam}\mathcal Q<\varepsilon_q(g)$ and $\mu(\partial \mathcal Q)=0$, for any $m\in\mathbb N^+$, one has
\begin{equation}
\begin{aligned}
h_{\mu}(f)\le \frac{1}{m}H_{\mu}(\mathcal Q^m)+\frac{d_u}{rq}(\int \log^+\left\|D_zf^q\right\|d\mu(z)+1)+\frac{\log 3qC_{r,d_u,d}}{q}+\log(d+1).   
\end{aligned}    
\end{equation}

\end{mainproposition}
\begin{mainproposition}\label{Proposition B}
Assume that $g\in\mathrm{Diff}^r(\mathbf M^d)\:(r>1)$, for any $q\in\mathbb N^+$, there exists a $C^r$ neighborhood $\mathcal V_g$ of $g$ and $\varepsilon_{q}(g)>0$, for any $f\in\mathcal V_g$, for any $\mu\in\mathcal M^{erg}(\mathbf M^d,f)$ with $\mathrm{dim}E_f^u(z)=1$,
for $\mu$-a.e.\:$z\in\mathbf M^d$, there exists a constant $C_{r,d}>0$, for any finite partition $\mathcal Q$ of $\mathbf M^d$ with $\mathrm{diam}\mathcal Q<\varepsilon_q(g)$ and $\mu(\partial \mathcal Q)=0$, for any $m\in\mathbb N^+$, one has
\begin{equation}
\begin{aligned}
h_{\mu}(f)\le \frac{1}{m}H_{\mu}(\mathcal Q^m)+\frac{1}{r-1}(\frac{1}{q}\int \log\left\|D_zf^q\right\|d\mu(z)-\lambda_1(f,\mu)+\frac{1}{q})+\frac{\log 3qC_{r,d}A_{f}^q}{q}+\log(d+1),   
\end{aligned}    
\end{equation}
where
$$A_f^q=q(\log^+\left\|Df\right\|_{0}+\log^+\left\|Df^{-1}\right\|_{0}+1).$$

\end{mainproposition}
\subsection{\texorpdfstring{The choices of $\varepsilon_{q}(g)$ and $\mathcal V_g$}{The choices of}}
Choose $\gamma>\max\left\{1,\left\|g\right\|_{r}\right\}$. There is a $C^r$ neighborhood $\mathcal V_{g}^1$ of $g$, such that for any $f\in\mathcal V_{g}^1$, $\left\|f\right\|_{r}< \gamma$. Choose $\varepsilon_q>0$ small enough and a $C^r$ neighborhood $\mathcal V_g^2$, so that for any $f\in\mathcal V_g^2$, for any $x,y\in \mathbf M^d$ with $\mathrm{d}(x,y)\le \varepsilon_q$, for any $1\le i\le q$, 
$$\left|\log\left\|D_xf^i\right\|-\log\left\|D_yf^i\right\|\right|\le 1.$$
Assume that $\mathcal V_{g}=\mathcal V_{g}^1\cap\mathcal V_{g}^2$ and $\varepsilon_q(g)=\min\left\{\varepsilon_{\gamma^i},\varepsilon_{\gamma^i}',\varepsilon_q,\gamma^{-i}:1\le i\le q\right\}$.

\subsection{Bounding the entropy along local unstable manifolds}
Assume that $f\in\mathcal V_g$ and 
$$\mu\in\mathcal M^{erg}(\mathbf M^d,f)\:\:\mathrm{with}\:\: \mathrm{dim}E^u_f(x)=d_u,$$ 
for $\mu$-a.e.\:$x\in\mathbf M^d$. Let $K$ be a compact subset of $\mathbf M^d$ with the following properties:\\

$\bullet\:\mu(K)>\frac{1}{2}$ and for any $\tau>0$, there exists $\rho>0$, such that for any $x\in K$, for any measurable set $\Sigma\subset W^u_{loc}(f,x)$ with $\mu_{x}(\Sigma\cap K)>0$, and for any finite partition $\mathcal P$ with $\mathrm{diam}\mathcal P<\rho$, one has
\begin{equation}\label{3.2}
h_{\mu}(f)\le \liminf_{n\to+\infty}\frac{1}{n}H_{\mu_{x,K,\Sigma}}(\mathcal P^n)+\tau,
\end{equation}
where $\mu_{x,K,\Sigma}(.)=\frac{\mu_x(.\:\cap K\cap \Sigma)}{\mu_x(K\cap \Sigma)}$,\\

$\bullet\:$ the following convergence holds uniformly for $x\in K$
\begin{equation}\label{3.3}
\frac{1}{n}\sum_{j=0}^{n-1}\delta_{f^jx}\to\mu,\:\frac{1}{n}\log\left\|D_xf^n|_{E^1_f(x)}\right\|\to\lambda_1(f,\mu),\:\mathrm{as}\:n\to+\infty,\\
\end{equation}

$\bullet\:$ for every $c\in \left\{0,1,\cdots q-1\right\}$, the following convergence holds uniformly for $x\in K$
\begin{equation}\label{3.41}
 \lim_{m\to+\infty}\frac{1}{m}\sum_{j=0}^{m-1}\log{\left\|D_{f^{qj+c}x}f^q\right\|}=\phi_c(x),   
\end{equation}
\begin{equation}\label{3.4}
 \lim_{m\to+\infty}\frac{1}{m}\sum_{j=0}^{m-1}\log^+{\left\|D_{f^{qj+c}x}f^q\right\|}=\phi_c^+(x),   
\end{equation}
where $\phi_c^+,\:\phi_c:\mathbf M^d\to\mathbb R$ are $f^q$-invariant measurable functions with 
$$\frac{1}{q}\sum_{c=0}^{q-1}\phi_c(x)=\int\log{\left\|D_zf^q\right\|d\mu(z)},$$
$$\frac{1}{q}\sum_{c=0}^{q-1}\phi_c^+(x)=\int\log^+{\left\|D_zf^q\right\|d\mu(z)}.$$
The existence of $K$ comes from the Egorov's theorem, Lemma \ref{Lemma 2.5}, and the Birkhoff ergodic theorem.   

\begin{lemma}\label{Lemma 3.1}
Assume that $x_0\in K$ and a $C^r$ embedded map $\sigma:[0,1]^{d_u}\to \mathbf M^d$ satisfying\\ $\bullet\:\mathrm{Im}\sigma \subset W^u_{loc}(f,x_0)$,\\ 
$\bullet\:\mu_{x_0}(\mathrm{Im}\sigma\cap K)>0$.\\ 
For any $\tau>0$, there exists $\rho>0$, for any finite partitions $\mathcal P$ with $\mathrm{diam}(\mathcal P)<\rho$ and $\mu(\partial\mathcal P)=0$, for any $m\in\mathbb N^+$, one has
$$h_{\mu}(f)\le \frac{1}{m}H_{\mu}(\mathcal Q^m)+\limsup_{n\to+\infty}\frac{1}{n}H_{\mu_{x_0,K,\mathrm{Im}\sigma}}(\mathcal P^n|\mathcal Q^n)+\tau.$$
\end{lemma}
\begin{proof}
By \eqref{3.3}, one has 
$$\frac{1}{n}\sum_{i=0}^{n-1}f^i_{*}\mu_{x_0,K,\mathrm{Im}\sigma}=\int_{K}\frac{1}{n}\sum_{i=0}^{n-1}\delta_{f^ix}d\mu_{x_0,K,\mathrm{Im}\sigma}\to\mu,$$
as $n\to+\infty$. Therefore, for any $m\in\mathbb N^+$,
 \begin{equation}\label{3.21}
 \begin{aligned}
 h_{\mu}(f)&\le \liminf_{n\to+\infty}\frac{1}{n}H_{\mu_{x_0,K,\mathrm{Im}\sigma}} (\mathcal P^n)+\tau\:\text{\color{blue}(diam$\mathcal P<\rho$)}\\
 &\le \limsup_{n\to+\infty}\frac{1}{n}H_{\mu_{x_0,K,\mathrm{Im}\sigma}} (\mathcal Q^n)+\limsup_{n\to+\infty}\frac{1}{n}H_{\mu_{x_0,K,\mathrm{Im}\sigma}}(\mathcal P^n|\:\mathcal Q^n)+\tau\\
 &\le \frac{1}{m}H_{\mu}(\mathcal Q^m)+\limsup_{n\to+\infty}\frac{1}{n}H_{\mu_{x_0,K,\mathrm{Im}\sigma}}(\mathcal P^n|\:\mathcal Q^n)+\tau\text{\color{blue}\:(Lemma \ref{Lemma 2.3})}.
 \end{aligned}    
 \end{equation} 
 \end{proof}
\subsection{Entropy and reparametrization lemma}
The Lebesgue covering dimension of a topological space $X$ is defined as follows: 
$$\mathrm{dim}_{\mathrm{Leb}}(X)$$ is the smallest integer $n\ge -1$ such that every open cover of $X$ has an open refinement with order at most $n+1$, where the order of a cover is the largest integer $m$ such that there exists a point in $X$ belonging to $m$ distinct sets of the cover. It is a basic fact that $\mathrm{dim}_{\mathrm{Leb}}(\mathbf M^d)=d$. Therefore, for any $\rho>0$, there exists an open cover $\mathcal V=\left\{U_1,\cdots, U_{k}\right\}$ of $\mathbf M^d$ with $\mathrm{diam}\mathcal V<\rho$, such that
$$\text{for any }x\in\mathbf M^d,\text{ there is at most } U_{n_{x,1}},\cdots,U_{n_{x,d+1}}\in\mathcal V,\text{ such that }x\in \cap_{1\le i\le d+1}U_{n_{x,i}}.$$
Let $\mathcal P$ be a finite partition of $\mathbf M^d$ satisfying
$$\mathrm{diam}\mathcal P<\rho,\:\mu(\partial\mathcal P)=0;$$
$$\mathcal P=\left\{V_1,\cdots,V_k\right\},\:\text{and for any}\:1\le i\le k,\:V_{i}\subset U_i.$$
Let $l,q\in\mathbb N^+$ and $0<\mathrm{diam}\mathcal Q<\varepsilon<\varepsilon_q(g)$.
Choose $x_{0}\in K$ such that there exists a $C^r$ embedded map $\sigma:[0,1]^{d_u}\to\mathbf M^d$ satisfying
$$\mathrm{Im}\sigma\subset W^u_{loc}(f,x_0);$$
$$\mu_{x_0}(\mathrm{Im}\sigma\cap K)>0;$$
$$\left\|\sigma\right\|_{r}\le\varepsilon.$$
If $d_u=1$, let $\sigma:[0,1]\to\mathbf M^d$ be a $C^r$ embedded curve satisfying
$$\mathrm{Im}\sigma\subset W^u_{loc}(f,x_0);$$
$$\mu_{x_0}(\mathrm{Im}\sigma\cap K)>0;$$
$$\sigma\text{ is strongly}\: \varepsilon\text{-bounded.}$$
\begin{definition}\cite[Definition 4.8, Definition 4.9]{BCS22}\\
(1) Let $\sigma:[0,1]^{l}\to \mathbf M^d$ be a $C^r$ embedded map. A reparametrization of $\sigma$ is a non-constant $C^r$ map $\psi:[0,1]^{l_{\psi}}\to [0,1]^{l}$ with $l_{\psi}\le l$. A family of reparametrizations of $\sigma$ over a subset $T\subset [0,1]^l$ is a collection $\mathcal R$ of reparametrizations such that $T\subset \cup_{\psi\in\mathcal R}\psi([0,1]^{l_{\psi}})$.\\
(2) A reparametrization $\psi$ of $\sigma$ is $(C^r,f,q,\varepsilon)$-admissible up to time $n$, if there exists an increasing sequence $(n_0,n_1,\cdots, n_t)$ such that\\
$\bullet\:n_0=0,n_{t}=n$ and $n_j-n_{j-1}\le q$ for any $1\le j\le t$,\\
$\bullet\:$for any $0\le j\le t$, $\left\|f^{n_j}\circ\sigma\circ\psi\right\|_{r}\le\varepsilon$.\\
We call the integers $n_j,\:1\le j\le t$ the admissible times. A family $\mathcal R$ of reparametrization of $\sigma$ over $T$, which is $(C^r,f,q,\varepsilon)$-admissible up to time $n$, if each $\psi\in\mathcal R$ is $(C^r,f,q,\varepsilon)$-admissible up to time $n$.
\end{definition}
Let 
\begin{equation}
\begin{aligned}
C_{n,q}(f)=\max_{F_n\in\mathcal Q^n}\min\left\{\lceil\frac{1}{n}\log\#\mathcal R_n\rceil:\mathcal R_n\:\text{is a family of reparametrization of $\sigma$ over}\right.\\
\left.\text{$\sigma^{-1}(F_n\cap K)$, which is $(C^r,f,q,\varepsilon)$-admissible up to time $n$}\right\}.
\end{aligned}
\end{equation}
The following lemma indicates the relationship between entropy and the cardinality of the family of reparametrization. Moreover, $C_{n,q}(f)$ is well-defined.
\begin{lemma}\label{Lemma 3.2}
For any $0<\varepsilon<\varepsilon_{q}(g)$, for any $n\in\mathbb N^+$, for any $F_n\in\mathcal Q^n$, there exists a family $\mathcal R_{n}$ of reparametrization of $\sigma$ over $\sigma^{-1}(F_n\cap K)$, which is $(C^r,f,q,\varepsilon)$-admissible up to time $n$. Moreover,
$$\limsup_{n\to+\infty}\frac{1}{n}H_{\mu_{x_0,K,\mathrm{Im}\sigma}}(\mathcal P^n|\:\mathcal Q^n)\le \limsup_{n\to+\infty}C_{n,q}(f)+\log (d+1).$$
\end{lemma}
\begin{proof}
For any $F_n\in\mathcal Q^n$, the existence of $\mathcal R$ of reparametrization of $\sigma$ over $\sigma^{-1}(F_n\cap K)$, which is $(C^r,f,q,\varepsilon)$-admissible up to time $n$ and the well-definedness of $C_{n,q}(f)$ hold. 

For any $n\in\mathbb N^+$, for any $F_n\in\mathcal Q^n$, 
choose $\mathcal R_{F_n}$ of reparametrization of $\sigma$ over $\sigma^{-1}(F_n\cap K)$, which is $(C^r,f,q,\varepsilon)$-admissible up to time $n$, satisfying
$$\max_{F_n\in\mathcal Q^n}\left\{\lceil\frac{1}{n}\log\#\mathcal R_{F_n}\rceil\right\}=C_{n,q}(f).$$
And it suffices to prove
$$ \limsup_{n\to+\infty}\frac{1}{n}H_{\mu_{x_0,K,\mathrm{Im}\sigma}}(\mathcal P^n|\:\mathcal Q^n)\le \limsup_{n\to+\infty}\frac{1}{n}\log\max_{F_n\in\mathcal Q^n}\#\mathcal R_{F_n}+\log(d+1).$$
By Jensen's inequality, we have
\begin{equation}\label{3.2-1}
\begin{aligned}
 \limsup_{n\to+\infty}\frac{1}{n}H_{\mu_{x_0,K,\mathrm{Im}\sigma}}(\mathcal P^n|\:\mathcal Q^n)&\le \limsup_{n\to+\infty}\frac{1}{n}\sum_{F_n\in\mathcal Q^n}\mu_{x_0,K,\mathrm{Im}\sigma}(F_n)\\
 &\log\#\left\{E_n\in\mathcal P^n:E_n\cap F_n\cap  K\cap\mathrm{Im}\sigma\ne\emptyset\right\}\\
 &\le \limsup_{n\to+\infty}\frac{1}{n}\log\max_{F_n\in\mathcal Q^n}\#\left\{E_n\in\mathcal P^n:E_n\cap F_n\cap  K\cap\mathrm{Im}\sigma\ne\emptyset\right\}.
\end{aligned}    
\end{equation}
 Then by \cite[Lemma 3.3]{B72} and the definition of $\mathcal P$, there exists $\delta>0$, for any $n\in\mathbb N^+$,
$$\max_{F_n\in\mathcal Q^n}\#\left\{E_n\in\mathcal P^n:E_n\cap F_n\cap  K\cap\mathrm{Im}\sigma\ne\emptyset\right\}\le \max_{F_n\in\mathcal Q^n}r_{n}( F_n\cap  K\cap\mathrm{Im}\sigma,\delta)(d+1)^n,$$
where $r_{n}(E,\delta)=\min\left\{\# \Lambda_n:\Lambda_n\:\mathrm{is}\:(n,\delta)\mathrm{-spanning}\:E\right\}$ for $E\subset \mathbf M^d$. Therefore, by \eqref{3.2-1}, we have
\begin{equation}\label{3.2-2}
 \limsup_{n\to+\infty}\frac{1}{n}H_{\mu_{x_0,K,\mathrm{Im}\sigma}}(\mathcal P^n|\:\mathcal Q^n)\le \limsup_{n\to+\infty}\frac{1}{n}\log \max_{F_n\in\mathcal Q^n}r_{n}( F_n\cap  K\cap\mathrm{Im}\sigma,\delta)+\log (d+1).     
\end{equation}
By the definition of reparametrization, for any $F_n\in\mathcal Q^n$, for any $\theta\in\mathcal R_{F_n}$, for any $0\le j\le n-1$, we have
$$\left\|D(f^j\circ\sigma\circ\theta)\right\|_{0}\le 1.$$
Let $\theta\in\mathcal R_{F_n}$, and let $\Psi_{\delta,\theta}$ be the closed cover of $[0,1]^{l'_{\theta}}$ by cubes of diameter less than $\frac{1}{\sqrt{d}}\delta$. Moreover, we assume that $\#\Psi_{\delta,\theta}=C_{\delta,l'_{\theta},d}$. Then $\Psi_{\delta,\theta}$ induces the family of affine maps $\mathcal R_{\delta,\theta}$ satisfying\\
$\bullet\:\#\mathcal R_{\delta,\theta}=\#\Psi_{\delta,\theta}$;\\
$\bullet\:[0,1]^{l'_{\theta}}= \cup_{\psi\in R_{\delta,\theta}}\psi([0,1]^{l'_{\theta}})$;\\
$\bullet\:$for any $\psi\in\mathcal R_{\delta,\theta}$, for any $0\le j\le n-1$, $\left\|D(f^j\circ\sigma\circ\theta\circ\psi)\right\|_{0}\le \frac{1}{\sqrt{d}}\delta.$\\
Moreover,
$$\sigma^{-1}(F_n\cap K)\subset \cup_{\theta\in\mathcal R_{F_n}}\theta([0,1]^{l'_{\theta}})\subset \cup_{\theta\in\mathcal R_{F_n}}\cup_{\psi\in\mathcal R_{\delta,\theta}}\theta\circ\psi([0,1]^{l'_{\theta}}).$$
Therefore, 
$$F_n\cap K\cap\mathrm{Im}\sigma\subset \cup_{\theta\in\mathcal R_{F_n}}\cup_{\psi\in\mathcal R_{\delta,\theta}}\sigma\circ\theta\circ\psi([0,1]^{l'_{\theta}}).$$
Thus, $\left\{s_{\sigma,\theta,\psi}\in \mathrm{Im(\sigma\circ\theta\circ\psi)}:\theta\in\mathcal R_{F_n},\psi\in\mathcal R_{\delta,\theta}\right\}$ is $(n,\delta)$-spanning $F_n\cap K\cap \mathrm{Im}\sigma$. And therefore
$$r_n(F_n\cap K\cap \mathrm{Im}\sigma,\delta)\le \max_{0\le i\le d}\left\{\#\Psi_{\sigma,i}\right\}\times \#\mathcal R_{F_n}.$$
By \eqref{3.2-2}, we have
$$\limsup_{n\to+\infty}\frac{1}{n}H_{\mu_{x_0,K,\mathrm{Im}\sigma}}(\mathcal P^n|\:\mathcal Q^n)\le \limsup_{n\to+\infty}\frac{1}{n}\log\max_{F_n\in\mathcal Q^n}\#\mathcal R_{F_n}+\log(d+1).$$
It implies that
$$\limsup_{n\to+\infty}\frac{1}{n}H_{\mu_{x_0,K,\mathrm{Im}\sigma}}(\mathcal P^n|\:\mathcal Q^n)\le \limsup_{n\to+\infty}C_{n,q}(f)+\log (d+1).$$
\end{proof}
\subsection{Yomdin's estimation}
\begin{lemma}\label{Lemma 3.3}
\begin{equation}
\begin{aligned}
\limsup_{n\to+\infty}C_{n,q}(f)&\le\frac{d_{u}}{qr}(\int\log^+{\left\|D_zf^q\right\|}d\mu(z)+1)+\frac{\log 3qC_{r,d_u,d}}{q}.
\end{aligned}
\end{equation}  
\end{lemma}
\begin{proof}
By the choice of $K$, the following convergence holds uniformly for $x\in K$
\begin{equation*}
\begin{aligned}
&\sum_{c=0}^{q-1}\lim_{m\to+\infty}\frac{1}{m}\sum_{j=0}^{m-1}\log^+\left\|D_{f^{qj+c}x}f^q\right\|=\lim_{n\to+\infty}\frac{q}{n}\sum_{j=0}^{n-1}\log^+\left\|D_{f^{j}x}f^q\right\|=q\int\log^+{\left\|D_zf^q\right\|}d\mu(z).
\end{aligned}
\end{equation*}
Hence, for every $x\in K$, there exists $c(x)\in\left\{0,1,\cdots,q-1\right\}$ such that
\begin{equation}\label{3.8}
\begin{aligned}
&\lim_{m\to+\infty}\frac{1}{m}\sum_{j=0}^{m-1}\log^+\left\|D_{f^{qj+c(x)}x}f^q\right\|\le \int\log^+{\left\|D_zf^q\right\|}d\mu(z).
\end{aligned}
\end{equation}
We decompose $K$ to be the union of $\left\{K_c\right\}_{0\le c\le q-1}$, such that for any $x\in K_c$, $c(x)=c$. For any $F_n\in \mathcal Q^n$, we decompose $F_n\cap K\cap \mathrm{Im}\sigma$ to be the union of $F_n^{c,\mathbf k}$, such that for any $x\in F_n^{c,\mathbf k}$, one has\\

$\bullet\:c(x)=c$;\\

$\bullet\:\lceil\log^+\left\|D_xf^c\right\|\rceil=k_{0};\\$

$\bullet\:\forall 1\le j\le [\frac{n-c}{q}],\:\lceil\log^+\left\|D_{f^{(j-1)q+c}x}f^q\right\|\rceil=k_{j};\\$

$\bullet\:\lceil\log^+\left\|D_{f^{[\frac{n-c}{q}]q+c}x}f^{n-([\frac{n-c}{q}]q+c)}\right\|\rceil=k_{[\frac{n-c}{q}]}$;\\

$\bullet\:\mathbf k=\left\{k_{j}\right\}_{0\le j\le [\frac{n-c}{q}]}\in \mathbb N^{[\frac{n-c}{q}]+1}$.\\
 
There are at most $q(3q)^{[\frac{n-c}{q}]+1}$ possible choices, such that $F_{n}^{c,\mathbf k}\ne\emptyset$.

Suitably choose $y\in F_{n}^{c,\mathbf k}\cap K\cap\mathrm{Im\sigma}$, satisfying
$$\sigma^{-1}(F_n^{c,\mathbf k}\cap K)\subset \sigma^{-1}(\left\{z\in F_n^{c,\mathbf k}:\mathrm{d}(f^{c+jq}(z),f^{c+jq}(y))\le\varepsilon,\forall\:0\le j\le[\frac{n-c}{q}]\right\}).$$
Therefore, by Lemma \ref{2.6}, according to admissible times $(c+q\mathbb N)\cup\left\{0,n\right\}\cap\left\{0,1,\cdots,n\right\}$, there exists a family $\Gamma_{F_n}^{c,\mathbf k}$ of reparametrization of $\sigma$ over $\sigma^{-1}(F_n^{c,\mathbf k}\cap K)$, which is $(C^r,f,q,\varepsilon)$-admissible up to time $n$, such that
\begin{equation}\label{11}
\begin{aligned}
\#\Gamma_{F_n}^{c,\mathbf k}&\le C_{r,d_{u},d}^{[\frac{n-c}{q}]+1}\mathrm{exp}(\frac{d_u}{r}\sum_{j=0}^{[\frac{n-c}{q}]+1}k_{j}).
\end{aligned}
\end{equation}
By \eqref{3.8} and \eqref{11}, there exists $x\in F_n\cap K\cap\mathrm{Im}\sigma$, such that
\begin{equation*}
\begin{aligned}
\limsup_{n\to+\infty}\frac{1}{n}\log\#\Gamma_{F_n}^{c,\mathbf k}&\le
\limsup_{n\to+\infty}\frac{1}{n}\frac{d_u}{r}\sum_{j=0}^{[\frac{n-c}{q}]}(\log^+\left\|D_{f^{qj+c}x}f^q\right\|+1)+\frac{\log C_{r,d_u,d}}{q}\\
&\le\frac{d_u}{rq}(\int\log^+{\left\|D_zf^q\right\|}d\mu(z)+1)+\frac{\log C_{r,d_u,d}}{q}.
\end{aligned}    
\end{equation*}
Therefore, for any $F_n\in\mathcal Q^n$, according to admissible times $(c(x)+q\mathbb N)\cup\left\{0,n\right\}\cap\left\{0,1,\cdots,n\right\}$, there exists a family $\mathcal R_{F_n}$ of reparametrization of $\sigma$ over $\sigma^{-1}(F_n\cap K)$, which is $(C^r,f,q,\varepsilon)$-admissible up to time $n$, such that
\begin{equation*}
\begin{aligned}
\limsup_{n\to+\infty}\frac{1}{n}\log\max_{F_n\in\mathcal Q^n}\#\mathcal R_{F_n}&\le\frac{d_u}{qr}(\int\log^+{\left\|D_zf^q\right\|}d\mu(z)+1)+\frac{\log 3qC_{r,d_u,d}}{q}.
\end{aligned}    
\end{equation*}
It implies that
\begin{equation*}
\begin{aligned}
\limsup_{n\to+\infty}C_{n,q}(f)&\le\frac{d_u}{qr}(\int\log^+{\left\|D_zf^q\right\|}d\mu(z)+1)+\frac{\log 3qC_{r,d_u,d}}{q}.
\end{aligned}    
\end{equation*} 
\end{proof}
\subsection{Burguet's estimation}
\begin{lemma}\label{Lemma 3.5}
Assume that $\mathrm{dim}E^u_{f}(z)=1$, for $\mu$-a.e.\:$z\in\mathbf M^d$. Then, one has
\begin{equation}
\begin{aligned}
\limsup_{n\to+\infty}C_{n,q}(f)&\le\frac{1}{r-1}(\frac{1}{q}\int\log{\left\|D_zf^q\right\|}d\mu(z)-\lambda_1(f,\mu)+\frac{1}{q})+\frac{\log 3qC_{r,d}A_{f}^q}{q},
\end{aligned}
\end{equation} 
where
$$A_f^q=q(\log^+\left\|Df\right\|_{0}+\log^+\left\|Df^{-1}\right\|_{0}+1).$$
\end{lemma}
\begin{proof}
By the choice of $K$, the following convergence holds uniformly for $x\in K$
\begin{equation*}
\begin{aligned}
&\sum_{c=0}^{q-1}\lim_{m\to+\infty}\frac{1}{m}\sum_{j=0}^{m-1}(\log\left\|D_{f^{qj+c}x}f^q\right\|-\log\left\|D_{f^{qj+c}x}f^q|_{E^u_f(f^{qj+c}x)}\right\|)\\
&=\lim_{n\to+\infty}\frac{q}{n}\sum_{j=0}^{n-1}(\log\left\|D_{f^{j}x}f^q\right\|-\log \left\|D_{f^{j}x}f^q|_{E^u_f(f^jx)}\right\|)\\
&=q(\int\log{\left\|D_zf^q\right\|}d\mu(z)-q\lambda_1(f,\mu)).
\end{aligned}
\end{equation*}
Hence, for every $x\in K$, there exists $c(x)\in\left\{0,1,\cdots,q-1\right\}$ such that
\begin{equation}\label{3.14}
\begin{aligned}
&\lim_{m\to+\infty}\frac{1}{m}\sum_{j=0}^{m-1}(\log\left\|D_{f^{qj+c(x)}x}f^q\right\|-\log \left\|D_{f^{qj+c(x)}x}f^q|_{E^u_f(f^{qj+c(x)}x)}\right\|)\\
&\le \int\log{\left\|D_zf^q\right\|}d\mu(z)-q\lambda_1(f,\mu).
\end{aligned}
\end{equation}
We decompose $K$ to be the union of $\left\{K_c\right\}_{0\le c\le q-1}$, such that for any $x\in K_c$, $c(x)=c$. For any $F_n\in \mathcal Q^n$, we decompose $F_n\cap K\cap \mathrm{Im}\sigma$ to be the union of $F_n^{c,\mathbf k,\mathbf k'}$, such that for any $x\in F_n^{c,\mathbf k,\mathbf k'}$, one has\\

$\bullet\:c(x)=c$;\\

$\bullet\:\lceil\log\left\|D_xf^c\right\|\rceil=k_{0},\:\lceil\log \left\|D_xf^c|_{E^u_{f}(x)}\right\|\rceil=k'_{0};\\$

$\bullet\:\forall 1\le j\le [\frac{n-c}{q}],\:\lceil\log\left\|D_{f^{(j-1)q+c}x}f^q\right\|\rceil=k_{j},\:\lceil\log \left\|D_{f^{(j-1)q+c}x}f^q|_{E^u_{f}(f^{(j-1)q+c}x)}\right\|\rceil=k'_{j};\\$

$\bullet\:\lceil\log\left\|D_{f^{[\frac{n-c}{q}]q+c}x}f^{n-([\frac{n-c}{q}]q+c)}\right\|\rceil=k_{[\frac{n-c}{q}]},\:\lceil\log \left\|D_{f^{[\frac{n-c}{q}]q+c}x}f^{n-([\frac{n-c}{q}]q+c)}|_{E^u_{f}(f^{[\frac{n-c}{q}]q+c}x)}\right\|\rceil=k'_{[\frac{n-c}{q}]};$\\

$\bullet\:\mathbf k,\mathbf k'\in \mathbb Z^{[\frac{n-c}{q}]+1}$.\\

There are at most $q(3qA_f^q)^{[\frac{n-c}{q}]+1}$ possible choices, such that $F_{n}^{c,\mathbf k,\mathbf k'}\ne\emptyset$, where
$$A_f^q=q(\log^+\left\|Df\right\|_{0}+\log^+\left\|Df^{-1}\right\|_{0}+1).$$

Suitably choose $y\in F_n\cap K\cap\mathrm{Im\sigma}$, satisfying
$$\sigma^{-1}(F_n^{c,\mathbf k,\mathbf k'}\cap K)\subset \sigma^{-1}(\left\{z\in F_n^{c,\mathbf k,\mathbf k'}:\mathrm{d}(f^{c+jq}(z),f^{c+jq}(y))\le\varepsilon,\forall\:0\le j\le [\frac{n-c}{q}]\right\}).$$
By Lemma \ref{Lemma B} and Lemma \ref{Lemma C}, according to admissible times $(c+q\mathbb N)\cup\left\{0,n\right\}\cap\left\{0,1,\cdots,n\right\}$, there exists a family $\Gamma_{F_n}^{c,\mathbf k,\mathbf k'}$ of reparametrization of $\sigma$ over $\sigma^{-1}(F_n^{c,\mathbf k,\mathbf k'}\cap K)$, which is $(C^r,f,q,\varepsilon)$-admissible up to time $n$, such that
\begin{equation}
\begin{aligned}
\#\Gamma_{F_n}^{c,\mathbf k,\mathbf k'}&\le C_{r,d}^{[\frac{n-c}{q}]+1}\mathrm{exp}(\frac{1}{r-1}\sum_{j=0}^{[\frac{n-c}{q}]+1}(k_{j}-k_{j}')).
\end{aligned}
\end{equation}
By \eqref{3.14} and the definition of $F_{n}^{c,\mathbf k,\mathbf k'}$, one has
\begin{equation*}
\begin{aligned}
\limsup_{n\to+\infty}\frac{1}{n}\log\#\Gamma_{F_n}^{c,\mathbf k,\mathbf k'}&\le\frac{1}{r-1}(\frac{1}{q}\int\log{\left\|D_zf^q\right\|}d\mu(z)-\lambda_1(f,\mu)+\frac{1}{q})+\frac{\log C_{r,d}}{q},
\end{aligned}    
\end{equation*}
uniformly with respect to $F_n\in\mathcal Q^n$. Therefore, according to admissible times $(c(x)+q\mathbb N)\cup\left\{0,n\right\}\cap\left\{0,1,\cdots,n\right\}$, there exists a family $\Gamma_{F_n}$ of reparametrization of $\sigma$ over $\sigma^{-1}(F_n\cap K)$, which is $(C^r,f,q,\varepsilon)$-admissible up to time $n$, such that
\begin{equation*}
\begin{aligned}
\limsup_{n\to+\infty}\frac{1}{n}\log\max_{F_n\in\mathcal Q^n}\#\Gamma_{F_n}&\le\frac{1}{r-1}(\frac{1}{q}\int\log{\left\|D_zf^q\right\|}d\mu(z)-\lambda_1(f,\mu)+\frac{1}{q})+\frac{\log 3qC_{r,d}A_{f}^q}{q}.
\end{aligned}    
\end{equation*}
It implies that
\begin{equation*}
\begin{aligned}
\limsup_{n\to+\infty}C_{n,q}(f)&\le\frac{1}{r-1}(\frac{1}{q}\int\log{\left\|D_zf^q\right\|}d\mu(z)-\lambda_1(f,\mu)+\frac{1}{q})+\frac{\log 3qC_{r,d}A_{f}^q}{q}.
\end{aligned}    
\end{equation*} 
\end{proof}
\subsection{Proof of Proposition A}
\begin{equation*}
\begin{aligned}
h_{\mu}(f)&\le \frac{1}{m}H_{\mu}(\mathcal Q^m)+\limsup_{n\to+\infty}\frac{1}{n}H_{\mu_{x_0,K,\mathrm{Im}\sigma}}(\mathcal P^n|\mathcal Q^n)+\tau\text{\:\color{blue}(Lemma \ref{Lemma 3.1})}\\
&\le\frac{1}{m}H_{\mu}(\mathcal Q^m)+\limsup_{n\to+\infty}C_{n,q}(f)+\log(d+1)+\tau\text{\:\color{blue}(Lemma \ref{Lemma 3.2})}\\
&\le \frac{1}{m}H_{\mu}(\mathcal Q^m)+\frac{d_u}{rq}(\int \log^+\left\|D_zf^q\right\|d\mu(z)+1)+\frac{\log 3qC_{r,d_u,d}}{q} +\log(d+1)+\tau\text{\:\color{blue}(Lemma \ref{Lemma 3.3})}.
\end{aligned}    
\end{equation*}
Until now, Proposition \ref{Proposition A} has been proved by arbitrariness in the choice of $\tau>0$.
\subsection{Proof of Proposition B}
If $\mathrm{dim}E^u_f(x)=1$, for $\mu$-a.e.\:$x\in\mathbf M^d$, one has
\begin{equation*}
\begin{aligned}
h_{\mu}(f)&\le \frac{1}{m}H_{\mu}(\mathcal Q^m)+\limsup_{n\to+\infty}\frac{1}{n}H_{\mu_{x_0,K,\mathrm{Im}\sigma}}(\mathcal P^n|\mathcal Q^n)+\tau\text{\:\color{blue}(Lemma \ref{Lemma 3.1})}\\
&\le\frac{1}{m}H_{\mu}(\mathcal Q^m)+\limsup_{n\to+\infty}C_{n,q}(f)+\log(d+1)+\tau\text{\:\color{blue}(Lemma \ref{Lemma 3.2})}\\
&\le \frac{1}{m}H_{\mu}(\mathcal Q^m)+\frac{1}{r-1}(\frac{1}{q}\int \log\left\|D_zf^q\right\|d\mu(z)-\lambda_1(f,\mu)+\frac{1}{q})+\frac{\log 3qC_{r,d}A_{f}^q}{q} +\log(d+1)+\tau\\
&\text{\:\color{blue}(Lemma \ref{Lemma 3.5})},
\end{aligned}    
\end{equation*}
where $$A_f^q=q(\log^+\left\|Df\right\|_{0}+\log^+\left\|Df^{-1}\right\|_{0}+1).$$
Until now, Proposition \ref{Proposition B} has been proved by arbitrariness in the choice of $\tau>0$.
\section{Proof of Theorems A, B and C}
\subsection{Discretization of the measures}
Given $z\in\mathbf M^d$, for a diffeomorphism $g:\mathbf M^d\to\mathbf M^d$, we denote
$$\mu_{z,g}:=\lim_{n\to+\infty}\frac{1}{n}\sum_{i=0}^{n-1}\delta_{g^iz},$$
if the above limit exists. For any $n\in\mathbb N^+$, by the definition of $d^{u}_{\mathrm{max}}(f_n,\mu_n)$, there exists a Borel set $\Lambda_n\subset O(\mu_n)$ with $\mu_n(\Lambda_n)=1$, satisfying
$$d^{u}_{\mathrm{max}}(f_n,\mu_n)=\max_{z\in\Lambda_n}\mathrm{dim}E^u_{f_n}(z).$$
For $\mu_n\in\mathcal M(\mathbf M^d,f_n)$, we consider the decomposition
$$\mu_n=\sum_{0\le i,j\le d}\beta_n^{i,j}\mu_n^{i,j},\:\beta_n^{i,j}\in [0,1],\sum_{0\le i,j\le d}\beta_n^{i,j}=1$$
such that\\

$\bullet\:\forall0\le i,j\le d$, $\mu_n^{i,j}$ are $f_n$-invariant probability measures;\\

$\bullet\:\forall\:0\le i,j\le d$, for $\mu_n^{i,j}$-a.e.\:$z\in\mathbf M^d$, $\mu_{z,f_n}$ has exactly $i$ positive Lyapunov exponents and $j$ negative  Lyapunov exponents;\\

$\bullet\:\forall\:0\le i,j\le d,\:\mu_{n}(\left\{z\in\Lambda_n:\:\mathrm{dim}E^u_{f_n}(z)=i,\:\mathrm{dim}E^u_{f_n^{-1}}(z)=j\right\})=\beta_{n}^{i,j};\\$

$\bullet\:\forall\:i>d^u_{\max}(f_n,\mu_n)$, one has $\beta_{n}^{i,j}=0$.\\

Without loss of generality, we assume that\\

$\bullet\:\forall\:0\le i,j\le d,\:\lim_{n\to+\infty}\beta^{i,j}_{n}=\beta^{i,j},\:\lim_{n\to+\infty}\mu_{n}^{i,j}=\mu^{i,j};$\\

$\bullet\:\mu=\sum_{0\le i,j\le d}\beta^{i,j}\mu^{i,j},\:\mu^{i,j}\in\mathcal M(\mathbf M^d,f)$, for $0\le i,j\le d$.

\subsection{Ergodic decomposition}
By the ergodic decomposition as in \cite[Lemma 3.2]{LY25-2}, for each $0\le i,j\le d$, for each $n$, there are\\

$\bullet\:$positive numbers $\alpha_{n,1}^{i,j},\cdots,\alpha_{n,N_n^{i,j}}^{i,j}\in [0,1]$ satisfying $\sum_{t=1}^{N_n^{i,j}}\alpha_{n,t}^{i,j}=1;$\\

$\bullet\:f_n$-ergodic measures $\mu_{n,1}^{i,j},\cdots,\mu_{n,N_n^{i,j}}^{i,j};\\$

such that\\

$\bullet\:\lim_{n\to+\infty}\sum_{t=1}^{N_n^{i,j}}\alpha_{n,t}^{i,j}\mu_{n,t}^{i,j}=\mu^{i,j};\\$

$\bullet\:\left|h_{\mu_n^{i,j}}(f_n)-\sum_{t=1}^{N_n^{i,j}}\alpha_{n,t}^{i,j}h_{\mu_{n,t}^{i,j}}(f_n)\right|\le \frac{1}{n};$\\

$\bullet\:\forall\:n\in\mathbb N^+,\:\forall\:0\le i,j\le d,\:\forall1\le t\le N_n^{i,j}$, $\mu_{n,t}^{i,j}$ has exactly $i$ positive Lyapunov exponents and exactly $j$ negative Lyapunov exponents;\\ 

$\bullet\:\left|\lambda_1^+(f_n,\mu^{i,j}_{n})-\sum_{t=1}^{N_n^{i,j}}\alpha^{i,j}_{n,t}\lambda_1^+(f_n,\mu_{n,t}^{i,j})\right|\le \frac{1}{n}$.
\subsection{Proof of Theorem A}
Let $f_{n}\xrightarrow{C^r}f\in \mathrm{Diff}^r(\mathbf M^d)\:(r>1),\:n\to+\infty,$
 which is
 as in Theorem \ref{Theorem A}.
 We assume that the finite partition $\mathcal Q$ of $\mathbf M^d$ satisfies
 $$\mathrm{diam}\mathcal Q<\varepsilon_q(f),$$
 $$\forall\:\nu\in\left\{\mu^{i,j}_{n,t},\mu^{i,j}:0\le i,j\le d,n\in\mathbb N^+,1\le t\le N_n^{i,j}\right\},\:\nu(\partial\mathcal Q)=0.$$ 
By Proposition \ref{Proposition A} and Lemma \ref{Lemma 2.2}, for any $0\le i\le d^u_{\max}(f_n,\mu_n)$, for any $0\le j\le d$, for any $m\in\mathbb N^+$, we have
\begin{equation}
\begin{aligned}
\limsup_{n\to+\infty}h_{\mu_n^{i,j}}(f_n)&=\limsup_{n\to+\infty}\sum_{t=1}^{N_n^{i,j}}\alpha_{n,t}^{i,j}h_{\mu_{n,t}^{i,j}}(f_n)\\
&\le \frac{1}{m}\limsup_{n\to+\infty}\sum_{t=1}^{N^{i,j}_{n}}\alpha^{i,j}_{n,t}H_{\mu_{n,t}^{i,j}}(\mathcal Q^m)\\
&+\lim_{n\to+\infty}\frac{i}{qr}(\int\log^+{\left\|D_zf_n^q\right\|}d(\sum_{1\le t\le N_{n}^{i,j}}\alpha_{n,t}^{i,j}\mu^{i,j}_{n,t})(z)+1)+\frac{\log3qC_{r,d_u,d}}{q}+\log{(d+1)}\\
&\le \frac{1}{m}\lim_{n\to+\infty}H_{\sum_{t=1}^{N_n^{i,j}}\alpha_{n,t}^{i,j}\mu_{n,t}^{i,j}}(\mathcal Q^m)\\
&+\frac{i}{qr}(\int\log^+{\left\|D_zf^q\right\|}d\mu^{i,j}(z)+1)+\frac{\log3qC_{r,d_u,d}}{q}+\log{(d+1)}\\
&\le \frac{1}{m}H_{\mu^{i,j}}(\mathcal Q^m)+\frac{i}{qr}(\int\log^+{\left\|D_zf^q\right\|}d\mu^{i,j}(z)+1)+\frac{\log3qC_{r,d_u,d}}{q}+\log{(d+1)}.
\end{aligned}
\end{equation}
As $m\to+\infty$ and $q\to+\infty$, by the arbitrariness of $\mathcal Q$ and Lemma \ref{Lemma 2.1}, one has
$$\limsup_{n\to+\infty}h_{\mu_n^{i,j}}(f_n)\le h_{\mu^{i,j}}(f)+\frac{i\lambda^+_{1}(f,\mu^{i,j})}{r}+\log{(d+1)}.$$
By the decomposition of $\mu_n$ as in section 4.1, we have 
$$\sum_{0\le i\le d^u_{\mathrm{max}}(f_n,\mu_n),\:0\le j\le d}\beta^{i,j}_{n}h_{\mu_n^{i,j}}(f_n)=h_{\mu_n}(f_n).$$
Therefore,
$$\limsup_{n\to+\infty}h_{\mu_n}(f_n)\le h_{\mu}(f)+\limsup_{n\to+\infty}d^u_{\max}(f_n,\mu_n)\frac{\lambda_{1}^+(f,\mu)}{r}+\log(d+1).$$
For $M\in\mathbb N^+$, replace $f_n \xrightarrow{C^r} f$ by $f_{n}^M\xrightarrow{C^r}f^M$, as $n\to+\infty$, we have
$$\limsup_{n\to+\infty}h_{\mu_n}(f_n)\le h_{\mu}(f)+\limsup_{n\to+\infty}d^u_{\max}(f_n,\mu_n)\frac{\lambda_1^+(f,\mu)}{r}+\frac{\log(d+1)}{M}.$$
Let $M\to+\infty$, we have
$$\limsup_{n\to+\infty}h_{\mu_n}(f_n)\le h_{\mu}(f)+\limsup_{n\to+\infty}d^u_{\max}(f_n,\mu_n)\frac{\lambda_1^+(f,\mu)}{r}.$$

If $\limsup_{n\to+\infty}d^u_{\max}(f_n,\mu_n)=1$, there exists a subsequence $\left\{n_{k}\right\}_{k\in\mathbb N^+}$, such that
for any $k\in\mathbb N^+$, one has $$d^u_{\max}(f_{n_k},\mu_{n_k})=1.$$
By the Ruelle inequality, without loss of generality, we assume that 
$$\limsup_{n\to+\infty}h_{\mu_n}(f_n)=\limsup_{k\to+\infty}h_{\mu_{n_{k}}}(f_{n_k}).$$
Therefore, for any $k\in\mathbb N^+$, for any $i>1$, one has $\beta_{n_k}^{i,j}=0$. If $i=0$, by the Ruelle inequality, for any $0\le j\le d$, one has
$$\limsup_{k\to+\infty}\beta_{n_{k}}^{0,j}h_{\mu_{n_k}^{0,j}}(f_{n_k})\le \beta^{0,j}h_{\mu^{0,j}}(f).$$
If $i=1$, by Proposition \ref{Proposition B} and Lemma \ref{Lemma 2.2}, for any $0\le j\le d$, for any $m\in\mathbb N^+$, we have
\begin{equation}
\begin{aligned}
\limsup_{k\to+\infty}\beta_{n_k}^{1,j}h_{\mu_{n_k}^{1,j}}(f_{n_k})&=\limsup_{k\to+\infty}\beta_{n_k}^{1,j}\sum_{t=1}^{N_{n_k}^{1,j}}\alpha_{n_k,t}^{1,j}h_{\mu_{n_k,t}^{1,j}}(f_{n_k})\\
&\le \frac{1}{m}\limsup_{k\to+\infty}\beta_{n_k}^{1,j}\sum_{t=1}^{N^{1,j}_{n_k}}\alpha^{1,j}_{n_k,t}H_{\mu_{n_k,t}^{1,j}}(\mathcal Q^m)+\frac{\log3qC_{r,d}A_{f}^q}{q}+\log{(d+1)}+\limsup_{k\to+\infty}\\
&\frac{\beta_{n_k}^{1,j}}{r-1}(\frac{1}{q}\int\log{\left\|D_zf_{n_k}^q\right\|}d(\sum_{1\le t\le N_{n_k}^{1,j}}\alpha_{n_k,t}^{1,j}\mu^{1,j}_{n_k,t})(z)-\sum_{1\le t\le N_{n_k}^{1,j}}\alpha_{n_k,t}^{1,j}\lambda_1^+(f_{n_k},\mu_{n_k,t}^{1,j})+\frac{1}{q})\\
&\le\frac{\beta^{1,j}}{m}\lim_{k\to+\infty}H_{\sum_{t=1}^{N_{n_k}^{1,j}}\alpha_{n_k,t}^{1,j}\mu_{n_k,t}^{1,j}}(\mathcal Q^m)+\frac{\log3qC_{r,d}A_{f}^q}{q}+\log{(d+1)}\\
&+\frac{1}{r-1}(\frac{1}{q}\int\log{\left\|D_zf^q\right\|}d\mu^{1,j}(z)-\liminf_{k\to+\infty}\beta_{n_k}^{1,j}\lambda_1^+(f_{n_k},\mu_{n_k}^{1,j})+\frac{1}{q})\\
&\le\frac{\beta^{1,j}}{m}H_{\mu^{1,j}}(\mathcal Q^m)+\frac{1}{r-1}(\frac{1}{q}\int\log{\left\|D_zf^q\right\|}d\mu^{1,j}(z)-\liminf_{k\to+\infty}\beta_{n_k}^{1,j}\lambda_1^+(f_{n_k},\mu_{n_k}^{1,j})+\frac{1}{q})\\
&+\frac{\log3qC_{r,d}A_{f}^q}{q}+\log{(d+1)}.
\end{aligned}
\end{equation}
Summing over $j$ before taking the limit as $k\to+\infty$, we obtain
\begin{equation*}
\begin{aligned}
\limsup_{k\to+\infty}\sum_{0\le j\le d}\beta_{n_k}^{1,j}h_{\mu_{n_k}^{1,j}}(f_{n_k})&\le \frac{1}{m}\sum_{0\le j\le d}\beta^{1,j}H_{\mu^{1,j}}(\mathcal Q^m)+\frac{1}{r-1}(\frac{1}{q}\int\log{\left\|D_zf^q\right\|}d(\sum_{0\le j\le d}\beta^{1,j}\mu^{1,j})(z)\\
&-\liminf_{k\to+\infty}\sum_{0\le j\le d}\beta_{n_k}^{1,j}\lambda_1^+(f_{n_k},\mu_{n_k}^{1,j})+\frac{1}{q})+\frac{\log3qC_{r,d}A_{f}^q}{q}+\log{(d+1)}.     
\end{aligned}    
\end{equation*}
By the decomposition of $\mu_n$ as in section 4.1, for $n\in\mathbb N^+$ large enough, we have 
$$\sum_{i=1,\:0\le j\le d}\beta^{i,j}_{n}h_{\mu_n^{i,j}}(f_n)=h_{\mu_n}(f_n),$$
$$\sum_{i=1,\:0\le j\le d}\beta^{i,j}_{n}\lambda_1^+(f_n,\mu_{n}^{1,j})=\lambda_1^+(f_n,\mu_n).$$
Therefore, as $m\to+\infty$ and $q\to+\infty$, by the arbitrariness of $\mathcal Q$ and Lemma \ref{Lemma 2.1}, one has
$$\limsup_{n\to+\infty}h_{\mu_n}(f_n)\le h_{\mu}(f)+\frac{1}{r-1}(\lambda_{1}^+(f,\mu)-\liminf_{n\to+\infty}\lambda_1^+(f_n,\mu_n))+\log{(d+1)}.$$
For $M\in\mathbb N^+$, replacing $f_n \xrightarrow{C^r} f$ by $f_{n}^M\xrightarrow{C^r}f^M$, as $n\to+\infty$, we have
$$\limsup_{n\to+\infty}h_{\mu_n}(f_n)\le h_{\mu}(f)+\frac{1}{r-1}(\lambda_{1}^+(f,\mu)-\liminf_{n\to+\infty}\lambda_1^+(f_n,\mu_n))+\frac{\log{(d+1)}}{M}.$$
As $M\to+\infty$, we have
$$\limsup_{n\to+\infty}h_{\mu_n}(f_n)\le h_{\mu}(f)+\frac{1}{r-1}(\lambda_{1}^+(f,\mu)-\liminf_{n\to+\infty}\lambda_1^+(f_n,\mu_n)).$$

\subsection{Proof of Corollarys}
\paragraph{\textbf{Proof of Corollary A}}
 For any $n\in\mathbb N^+$, let
$\mu_{n}=\sum_{0\le i,j\le d}\beta^{i,j}_{n}\mu_{n}^{i,j}$ and $\mu=\sum_{0\le i,j\le d}\beta^{i,j}\mu^{i,j}$ satisfy\\

$\bullet\:\forall\:0\le i,j\le d$, $\mu_n^{i,j}$ is an $f_n$-invariant probability measure,\:$\mu^{i,j}$ is an $f$-invariant probability measure;\\

$\bullet\:\forall\:0\le i,j\le d$, for $\mu_n^{i,j}$-a.e.\:$z\in\mathbf M^d$, $\mu_{z,f_n}$ has exactly $i$ positive Lyapunov exponents and $j$ negative  Lyapunov exponents;\\

$\bullet\:\forall\:0\le i,j\le d,\:\beta^{i,j},\beta^{i,j}_{n}\in [0,1],\:\sum_{0\le i,j\le d}\beta^{i,j}_{n}=1,\:\sum_{0\le i,j\le d}\beta^{i,j}=1;\\$

$\bullet\:\forall\:0\le i,j\le d,\:\lim_{n\to+\infty}\beta^{i,j}_{n}=\beta^{i,j},\:\lim_{n\to+\infty}\mu_{n}^{i,j}=\mu^{i,j}.\\$

If $d=2$, one has
\begin{equation}
\begin{aligned}
\limsup_{n\to+\infty}h_{\mu_n}(f_n)&\le \sum_{0\le j\le 2}\beta^{0,j}\limsup_{n\to+\infty}h_{\mu_n^{0,j}}(f_n)+\sum_{1\le i\le 2}\beta^{i,0}\limsup_{n\to+\infty}h_{\mu_n^{i,0}}(f_n)+\beta^{1,1}\limsup_{n\to+\infty}h_{\mu_n^{1,1}}(f_n)\\
&\le \beta^{1,1}\limsup_{n\to+\infty}h_{\mu_n^{1,1}}(f_n)\mathbf{\color{blue}\:(Ruelle\:inequality)}\\
&\le \beta^{1,1}(h_{\mu^{1,1}}(f)+\frac{\min\left\{\lambda^+_{1}(f,\mu^{1,1}),\lambda^+_{1}(f^{-1},\mu^{1,1})\right\}}{r})\mathbf{\color{blue}\:(Theorem\:\ref{Theorem A})}\\
&\le h_{\mu}(f)+\frac{\lambda^+_{\max}(f,\mu)}{r}.
\end{aligned}    
\end{equation}

If $d>2$, it implies that for any $n\in\mathbb N^+$, if $\mu^{i,j}_{n}$ is well-defined, one has $i\le [\frac{d}{2}]$ or $j\le [\frac{d}{2}]$. 
Then we have
\begin{equation}
\begin{aligned}
\limsup_{n\to+\infty}h_{\mu_n}(f_n)&\le \sum_{0\le i,j\le d}\beta^{i,j}\limsup_{n\to+\infty}h_{\mu_{n}^{i,j}}(f_n)\\
&\le \sum_{0\le i,j\le d,\:i\le [\frac{d}{2}]}\beta^{i,j}(h_{\mu^{i,j}}(f)+[\frac{d}{2}]\frac{\lambda^+(f,\mu^{i,j})}{r})\\
&+\sum_{0\le i,j\le d,\:j\le [\frac{d}{2}],\:i>[\frac{d}{2}]}\beta^{i,j}(h_{\mu^{i,j}}(f)+[\frac{d}{2}]\frac{\lambda^+(f^{-1},\mu^{i,j})}{r})\:\mathbf{\color{blue}\:(Theorem\:\ref{Theorem A})}\\
&\le h_{\mu}(f)+[\frac{d}{2}]\frac{\lambda_{\max}^+(f,\mu)}{r},
\end{aligned}
\end{equation}
by considering $f_{n}\xrightarrow{C^r} f,\:\mu_{n}^{i,j}\to\mu^{i,j}\:(n\to+\infty)$ for all $0\le i\le [\frac{d}{2}]$  and $f_{n}^{-1}\xrightarrow{C^r} f^{-1},\:\mu_{n}^{i,j}\to\mu^{i,j}\:(n\to+\infty)$ for all $0\le j\le [\frac{d}{2}]$.

By the variational principle, for any $n\in\mathbb N^+$, there exists $\mu_n\in\mathcal M(\mathbf M^d,f_n)$, such that
$$\limsup_{n\to+\infty}h_{\mathrm{top}}(f_n)=\limsup_{n\to+\infty}h_{\mu_n}(f_n).$$
 By $f_{n}\xrightarrow{C^r}f$, as $n\to+\infty$, we assume that $\mu_n\to\nu\in\mathcal M(\mathbf M^d,f)$, as $n\to+\infty$. Therefore, we have
$$\limsup_{n\to+\infty}h_{\mathrm{top}}(f_n)\le h_{\nu}(f)+[\frac{d}{2}]\frac{\lambda^+_{\max}(f,\nu)}{r}\le h_{\mathrm{top}}(f)+[\frac{d}{2}]\frac{\sup_{\mu\in\mathcal M(\mathbf M^d,f)}\lambda^+_{\max}(f,\mu)}{r}.$$

\paragraph{\textbf{Proof of Corollary B}} 
There is a claim as follows.
\begin{Claim}\label{Claim 4.1}
With the assumptions in Corollary \ref{Corollary B}, for any $\mu\in\mathcal M(\Lambda,f)$, there exist $1\ge \beta\ge 0$ and a decomposition of $f$-invariant measure $\mu=\beta\nu_{1}+(1-\beta)\nu_{0}$, such that\\
$(1)\:h_{\nu_{0}}(f)=0$,\\
$(2)\:$for $\nu_1$-a.e.\:$z\in\mathbf M^d$, there exists exactly one positive Lyapunov exponent.
\end{Claim}
\begin{proof}
Let $$\Lambda_1=\left\{z\in\Lambda:\liminf_{n\to\pm\infty}\frac{1}{n}\log\left\|D_zf^n|_{E^{cu}(z)}\right\|>0\right\},$$
$$\beta=\mu(\Lambda_1).$$
Then there exists $\nu_0:=\mu|_{\Lambda\setminus\Lambda_1},\nu_1:=\mu|_{\Lambda_1}\in\mathcal M(\Lambda,f)$, such that $\mu=\beta\nu_1+(1-\beta)\nu_0$. By the Oseledets theorem and the Ruelle inequality, $h_{\nu_0}(f)=0$. Moreover, for $\nu_1$-a.e.\:$z\in\mathbf M^d$, there exists exactly one positive Lyapunov exponent. 

\end{proof}
Without loss of generality, we assume that for any $n\in\mathbb N^+$, $h_{\mu_n}(f)>0$. For any $n\in\mathbb N^+$, let
$\mu_{n}=\beta^{0}_{n}\mu_{n}^{0}+\beta^{1}_{n}\mu_{n}^{1}$ and $\mu=\beta^{0}\mu^{0}+\beta^{1}\mu^{1}$ satisfy\\

$\bullet\:$$\mu_n^{i},\mu^i\in\mathcal M(\Lambda,f)$ for $i=0,1$;\\

$\bullet\:h_{\mu_{n}^0}(f)=0;\\$

$\bullet\:$for $\mu_n^{1}$-a.e.\:$z\in\mathbf M^d$, $\mu_{z,f}$ has exactly one positive Lyapunov exponent and $\mu_n^1(\Lambda_1)=1$;\\

$\bullet\:\beta^{i}_{n},\beta^{i}\in [0,1],\:\beta^{0}_{n}=1-\beta^1_n,\:\beta^0=1-\beta^1$ for $i=0,1$;\\

$\bullet\:\lim_{n\to+\infty}\beta^{i}_{n}=\beta^{i},\:\lim_{n\to+\infty}\mu_{n}^{i}=\mu^{i}$ for $i=0,1$.\\

By the uniqueness of the Oseledets splitting, for any $\nu\in\mathcal M(\Lambda,f)$, for $\nu$-a.e.\:$z\in\Lambda$, one has $E^{cu}(z)=E^1_{f}(z)$ and for $\nu$-a.e.\:$z\in\Lambda_1$, one has $E^{cu}(z)=E^1_{f}(z)=E^u_{f}(z)$. By the continuity of $z\mapsto E^{cu}(z)$, $z\mapsto \log\left\|D_zf|_{E^{cu}(z)}\right\|$ is continuous on $\Lambda$. By the Birkhoff ergodic theorem and the boundedness of $z\mapsto \log\left\|D_zf|_{E^{cu}(z)}\right\|$ on $\Lambda_1$, one has 
$$\lim_{n\to+\infty}\lambda_{1}(f,\mu_n^1)=\lambda_1(f,\mu^1).$$
For $\mu^1\in\mathcal M(\Lambda,f)$, one has $\lambda_1(f,\mu^1)=\lambda_1^+(f,\mu^1)$. And therefore,
$$\lim_{n\to+\infty}\lambda_{1}^+(f,\mu_n^1)=\lambda_1^+(f,\mu^1).$$
By \eqref{1.2}, we have
$$\limsup_{n\to+\infty}h_{\mu^{1}_n}(f)\le h_{\mu^1}(f).$$
By Claim \ref{Claim 4.1}, one has
$$\limsup_{n\to+\infty}h_{\mu_n}(f)\le h_{\mu}(f).$$
\paragraph{\textbf{Proof of Corollary C}}There is a claim as follows.
\begin{Claim}\label{Claim 4.2}
With the assumptions of Corollary \ref{Corollary C}, assume that there is a sequence $\left\{\mu_n\right\}_{n\in\mathbb N^+}$ of $f$-invariant measures converging to an $f$-invariant measure $\mu$, satisfying
$$\lambda^+_{\Sigma}(f,\mu_n)\to\lambda^+_{\Sigma}(f,\mu),\:\mathrm{as}\:n\to+\infty,$$
where $\lambda^+_{\Sigma}(f,\mu)=\int\sum_{\lambda_i(f,x)>0}\lambda_i(f,x)\mathrm{dim}E^i_f(x)d\mu(x)$. Then, one has
\begin{equation*}
\limsup_{n\to+\infty}h_{\mu_n}(f_n)\le h_{\mu}(f).
\end{equation*}    
\end{Claim}
\begin{proof}
For any $n\in\mathbb N^+$, let
$\mu_{n}=\sum_{i=\pm,j=0,1}\beta^{i,j}_{n}\mu_{n}^{i,j}$ and $\mu=\sum_{i=\pm,j=0,1}\beta^{i,j}\mu^{i,j}$ satisfy\\
$\bullet\:$$\mu_n^{i,j},\mu^{i,j}$ are $f$-invariant probability measures for $i=\pm,\:j=0,1$;\\
$\bullet\:$for $\mu_n^{+,j}$-a.e.\:$z\in\mathbf M^d$, $\mu_{z,f}$ has exactly $j$ positive Lyapunov exponent, for $j=0,1$;\\
$\bullet\:$for $\mu_n^{-,j}$-a.e.\:$z\in\mathbf M^d$, $\mu_{z,f}$ has exactly $j$ negative Lyapunov exponent, for $j=0,1$;\\
$\bullet\:\beta^{i,j}_{n},\beta^{i,j}\in [0,1],\:\sum_{i=\pm,j=0,1}\beta_{n}^{i,j}=1,\:\sum_{i=\pm,j=0,1}\beta_{n}^{i,j}=1$ for $i=\pm,\:j=0,1$;\\
$\bullet\:\lim_{n\to+\infty}\beta^{i,j}_{n}=\beta^{i,j},\:\lim_{n\to+\infty}\mu_{n}^{i,j}=\mu^{i,j}$ for $i=\pm,\:j=0,1$.

By the Ruelle inequality, one has
$$\limsup_{n\to+\infty}h_{\mu_n^{i,0}}(f_n)\le h_{\mu^{i,0}}(f),\:i=\pm.$$
By the continuity of $x\mapsto \log\left|\mathrm{det}(D_xf)\right|$ and the Birkhoff ergodic theorem, if $\mu_n\to\mu,\:\lambda^+_{\Sigma}(f,\mu_n)\to\lambda^+_{\Sigma}(f,\mu)$, as $n\to+\infty$, one has $\lambda^+_{\Sigma}(f^{-1},\mu_n)\to\lambda^+_{\Sigma}(f^{-1},\mu)$, as $n\to+\infty$. We give a formula for $\lambda^+_{\Sigma}(f,\mu)$
$$\lambda^+_{\Sigma}(f,\mu)=\lim_{n\to+\infty}\frac{1}{n}\int\max_{1\le k\le d}\log^+\left\|\wedge^kD_xf^n\right\|d\mu(x).$$
It is clear that $\phi_n(x)=\max_{1\le k\le d}\log^+\left\|\wedge^kD_xf^n\right\|$ is continuous on $\mathbf M^d$ and $\left\{\phi_n\right\}_{n\in\mathbb N}$ is a sequence of sub-additive functions. By Kingman's sub-additive ergodic theorem, one has that the map $\mu\mapsto \lambda^+_{\Sigma}(f,\mu)$ is upper semicontinuous.

If $\beta^{i,1}=0$, it is clear that 
$$\limsup_{n\to+\infty}\beta_{n}^{i,1}h_{\mu_n^{i,1}}(f)\le \beta^{i,1}h_{\mu^{i,1}}(f).$$
Therefore, without loss of generality, we assume that $\beta^{i,1}>0$, for $i=\pm$.
Moreover, we assume that $\lim_{n\to+\infty}\lambda^+_{\Sigma}(f,\mu_n^{i,j})$ exists, for all $i=\pm,\:j=0,1.$ Thus
\begin{equation*}
\begin{aligned}
\lambda^+_{\Sigma}(f,\mu)&=\lim_{n\to+\infty}\lambda^+_{\Sigma}(f,\mu_n)\\
&=\lim_{n\to+\infty}\sum_{i=\pm,j=0,1}\beta^{i,j}_{n}\lambda^+_{\Sigma}(f,\mu_n^{i,j})\\
&\le \sum_{i=\pm,j=0,1}\beta^{i,j}\lim_{n\to+\infty}\lambda^+_{\Sigma}(f,\mu_n^{i,j})\\
&\le \sum_{i=\pm,j=0,1}\beta^{i,j}\lambda^+_{\Sigma}(f,\mu^{i,j})=\lambda^+_{\Sigma}(f,\mu). 
\end{aligned}    
\end{equation*}
It implies that the equality must hold, and then we have
$$\lim_{n\to+\infty}\lambda^+_{\Sigma}(f,\mu_n^{+,1})=\lambda^+_{\Sigma}(f,\mu^{+,1}).$$
Therefore
\begin{equation*}
\begin{aligned}
\lambda_{\Sigma}^+(f,\mu^{+,1})=\lim_{n\to+\infty}\lambda^+_{\Sigma}(f,\mu_n^{+,1})=\lim_{n\to+\infty}\lambda_1^+(f,\mu_n^{+,1})\le \lambda_1^+(f,\mu^{+,1})\le \lambda^+_{\Sigma}(f,\mu^{+,1}),    
\end{aligned}    
\end{equation*}
which implies that
$$\lim_{n\to+\infty}\lambda_1^{+}(f,\mu_n^{+,1})=\lambda_1^+(f,\mu^{+,1}).$$
Similarly, 
$$\lim_{n\to+\infty}\lambda_1^{+}(f^{-1},\mu_n^{-,1})=\lambda_1^+(f^{-1},\mu^{-,1}).$$
By the definitions of $\mu_{n}^{i,j},\:i=\pm,\:j=0,1$ and \eqref{1.2} of Theorem \ref{Theorem A}, one has
$$\limsup_{n\to+\infty}\beta_{n}^{i,1}h_{\mu_n^{i,1}}(f)\le \beta^{i,1}h_{\mu^{i,1}}(f),\:i=\pm.$$
Above all,
$$\limsup_{n\to+\infty}h_{\mu_n}(f)\le h_{\mu}(f).$$
\end{proof}

By the upper semicontinuity of the map $\mu\mapsto \lambda^+_{\Sigma}(f,\mu)$, there exists a dense $G_{\delta}$ subset $\mathcal M$ of $\mathcal M(\mathbf M^d,f)$, such that for any $\nu\in\mathcal M$, the map $\mu\mapsto \lambda^+_{\Sigma}(f,\mu)$ is continuous at $\nu$, which implies that the entropy map $\mu\mapsto h_{\mu}(f)$ is upper semicontinuous at $\nu$ by Claim \ref{Claim 4.2}.
\subsection{Proof of Theorem B} 
The idea of the proof in this section is drawn from \cite[Theorem 6.1]{DN05} and \cite{LSW15}. For the sake of completeness, we outline the proof below.
Let \(\Lambda_f\) denote the wild hyperbolic set associated with
\(f\in \mathcal N\). Let \(\mathcal H_*(f)\) be the set of all hyperbolic
periodic points that are homoclinically related to \(\Lambda_f\). For $p\in\mathcal H_*(f)$, denote $\chi(p)=\frac{1}{\mathrm{per}(p)}\min\left\{\log \left|\chi_s^{-1}(p)\right|,\log \left|\chi_u(p)\right|\right\}>0$, where $\left|\chi_s(p)\right|<1$ and $\left|\chi_u(p)\right|>1$ are the norms of the two eigenvalues of $D_pf^{\mathrm{per}(p)}$ respectively, and $\mathrm{per}(p)$ the least period of $p$. 
Given \(f\in \mathcal N\) and \(p\in \mathcal H_*(f)\), let
\(\mathcal U_p(f)\) be a neighborhood of \(f\) such that, for every
\(g\in \mathcal U_p(f)\), the continuation \(p_g\) of \(p\) is well defined and
the continuation \(\Lambda_f(g)\) of \(\Lambda_f\) is a wild hyperbolic set for
\(g\). We say that a diffeomorphism \(g\in \mathcal N\) admits a basic set with large entropy near $\delta_{p_g}^{\mathrm{per}(p_g)}$ at scale $\frac{1}{n}$,  if there exists a basic set \(\Lambda(p_g,n)\) for \(g\)
such that the following conditions hold.
\begin{enumerate}
    \item There exists an ergodic measure
    \[
        \nu\in \mathcal M^{erg}\bigl(\Lambda(p_g,n),g\bigr)
    \]
    such that
    \[
        h_\nu(g)>\frac{\chi(p_g)}{r}-\frac{1}{n}.
    \]

    \item For every ergodic measure
    \[
        \mu\in \mathcal M^{erg}\bigl(\Lambda(p_g,n),g\bigr),
    \]
    one has
    \[
        \rho\bigl(\mu,\delta_{p_g}^{\mathrm{per}(p_g)}\bigr)<\frac{1}{n},
    \]
    where $\delta_{p_g}^{\mathrm{per}(p_g)}$ denotes the periodic measure induced by $p_g$.
\end{enumerate}
Let \(\mathcal S_n(p)\subset \mathcal U_p(f)\) denote the set of all
diffeomorphisms \(g\in \mathcal U_p(f)\) for which there exists a basic set with large entropy near $\delta_{p_g}^{\mathrm{per}(p_g)}$ at scale $\frac{1}{n}$.  We aim to prove that,
\[
    \bigcup_{m\in\mathbb N^+}\bigcap_{n\in\mathbb N^+}\mathcal S_n(p_{f_m})
\]
is \(C^r\)-residual in \(\mathcal N\), where $\left\{f_m\right\}_{m\in\mathbb N}$ is a countable dense subset of $\mathcal N$. It suffices to prove that, for any \(p\in \mathcal H_*(f)\), the set
\(\mathcal S_n(p)\) is \(C^r\)-open and dense in \(\mathcal U_p(f)\), by gluing a local residual subsets to construct a global residual set on $\mathcal N$. It is clear that \(\mathcal S_n(p)\) is \(C^r\)-open. Thus, it
remains only to prove \(C^r\)-density. More precisely, we shall show that, for
any \(f\in \mathcal N\), and any \(p\in \mathcal H_*(f)\), there exists
\(g\in \mathcal N\) arbitrarily \(C^r\)-close to \(f\) such that, there exists a basic set with large entropy near $\delta_{p_g}^{\mathrm{per}(p_g)}$ at scale $\frac{1}{n}$. In what follows, we assume that \(p\) is dissipative and
nonresonant. For simplicity of exposition, we also assume that \(p\) is a fixed
point.

We say that a hyperbolic periodic point \(p\) has an interval of homoclinic
tangencies if there exists a continuous map
\[
    \gamma:[0,1]\to W^s(f,p)\cap W^u(f,p)
\]
such that, for any \(t\in[0,1]\), \(\gamma(t)\) is a homoclinic tangency. Then, as in \cite{DN05}, there is a basic fact as follows. For every
hyperbolic periodic point \(p\) homoclinically related to the wild hyperbolic
set \(\Lambda_f\), there exists \(g\in \mathcal N\) arbitrarily \(C^r\)-close to
\(f\) such that the continuation \(p_g\) has an interval of homoclinic
tangencies.

By a small \(C^r\)-perturbation, we may assume that, there exists an interval
\[
    I\subset W^s_{loc}(f,p)
\]
consisting entirely of homoclinic tangencies between \(W^u(f,p)\) and
\(W^s(f,p)\). We may further assume that \(I\) is contained in \(U(p)\), where $U(p)$ denotes the neighborhood of $p$ with a sufficiently small diameter that has a local product structure with respect to $f$. We choose $C^r$-linearizing coordinates near \(p\) on $U(p)$ such that the local stable
leaves are horizontal and the local unstable leaves are vertical. Let $z_1=(a_1,0),\:z_2=(a_2,0)$
be the endpoints of \(I\), with \(0<a_1<a_2\). Given \(\varepsilon>0\) and a
sufficiently large integer \(N\in \mathbb N\), set
\[
    A(N)=\frac{\varepsilon(a_2-a_1)}{N^r},
    \:
    c=\frac{a_1+a_2}{2}.
\]

We now take a small $C^r$-perturbation $g_{N,\varepsilon}$
of \(f\), together with an interval
\[
    \widetilde I\subset W^u_{loc}(g_{N,\varepsilon},p),
\]
such that the following properties hold:
\begin{enumerate}
    \item \(p\) remains a hyperbolic fixed point of \(g_{N,\varepsilon}\);

    \item $g_{N,\varepsilon}=f
        \:\text{on}\:
        U(p)\cap f^{-1}(U(p))\cap f(U(p));$

    \item\:$W^{i}_{loc}(g_{N,\varepsilon},p)
        =
        W^{i}_{loc}(f,p)$, for $i=u$ or $s$;

    \item $I\subset W^s_{loc}(g_{N,\varepsilon},p);$

    \item There exists a positive integer \(N_*\) such that
    \[
        J_{N,\varepsilon}
        :=
        g_{N,\varepsilon}^{N_*}(\widetilde I)
    \]
    is given by
    \[
        J_{N,\varepsilon}
        =
        \left\{
        (x,y)\in U(p):
        x\in [a_1,a_2],\;
        y=A(N)\cos\left(
            \frac{\pi N(x-c)}{a_2-a_1}
        \right)
        \right\}.
    \]
\end{enumerate}
Moreover, one can choose the perturbation so that there exists a constant
\( K>0\), independent of \(N\) and \(\varepsilon\), satisfying
\[
    \|g_{N,\varepsilon}-f\|_{r}
    \le
     K\varepsilon .
\]
For each positive integer \(k\), define
\[
    U_k^u:=\bigcap_{j=0}^{k} f^j(U(p)),
    \qquad
    U_k^s:=\bigcap_{j=0}^{k} f^{-j}(U(p)).
\]
We shall only consider sufficiently large integers \(N\). Recall that, by the
construction above, the intersection \(I\cap J\) consists of precisely \(N\)
points. By \cite[p. 155]{PT93} and \cite[Section 5, p.481-483]{DN05}, We can choose a positive integer \(k=k(N)\) and a curvilinear rectangle
\(D_N\) in a small neighborhood of \(I\) with the following properties:
\begin{enumerate}
    \item The boundary \(\partial D_N\) consists of curves contained in leaves
    of foliations \(\mathcal F^s\) and \(\mathcal F^u\);

    \item $D_N\subset U_k^s\setminus U_{k+1}^s$;

    \item
        $g^k_{N,\varepsilon}(D_N)\subset U_k^u\setminus U_{k+1}^u;$

    \item The intersection $D_N\cap g^{N_*+k}_{N,\varepsilon}(D_N)$
    consists of \(N\) full-height curvilinear subrectangles of \(D_N\);

    \item Let
    \[
        \widetilde\Lambda_N
        :=
        \bigcap_{i\in\mathbb Z}
        g^{i(N_*+k)}_{N.\varepsilon}(D_N)
    \]
    be the maximal \(g^{N_*+k}_{N,\varepsilon}\)-invariant set contained in \(D_N\). Then
    \(\widetilde\Lambda_N\) is a basic set, and $\bigl(\widetilde\Lambda_N, g^{N_*+k}_{N,\varepsilon}\bigr)$    is topologically conjugate to the full shift on \(N\) symbols;

    \item Define
    \[
        \Lambda(p,N)
        :=
        \bigcup_{j=0}^{N_*+k-1} g^j_{N,\varepsilon}(\widetilde\Lambda_N).
    \]
    Then, for all sufficiently large \(N\), depending on \(n\), the set
    \(\Lambda(p,N)\) is a hyperbolic set for \(g_{N,\varepsilon}\);

    \item $\mathrm{dist}(\mathcal F^s(g^{-k}_{N,\varepsilon}z),W^s_{loc}(g_{N,\varepsilon},p)))\le A(N)\sim e^{-\chi(p)k},$ where $z=(0,y)$ be the right endpoints of $\tilde I=g^{-N_*}_{N,\varepsilon}(J)$.
\end{enumerate}

Since $\bigl(\widetilde\Lambda_N,g_{N,\varepsilon}^{N_*+k}\bigr)$
is topologically conjugate to the full shift $(\Sigma_N,\sigma)$
on \(N\) symbols, the entropy of the \(g_{N,\varepsilon}\)-invariant set \(\Lambda(p,N)\) is
given by
\[
    h_{\mathrm{top}}\bigl(g_{N,\varepsilon},\Lambda(p,N)\bigr)
    =
    \frac{\log N}{N_*+k} .
\]
We now estimate the asymptotic size of \(k=k(N)\). By the basic property of $U(p)$ and the definition of \(A(N)\), one has $N^{-r}\sim A(N)\sim e^{-\chi (p)k(N)}.$
Thus $N\sim e^{\chi (p)k(N)/r}.$
It follows that
\[
    \lim_{N\to+\infty}
    h_{\mathrm{top}}\bigl(g_{N,\varepsilon},\Lambda(p,N)\bigr)
    =
    \lim_{N\to+\infty}
    \frac{\log N}{N_*+k(N)}
    =
    \frac{\chi(p)}{r}.
\]
Since \(\Lambda(p,N)\) admits an ergodic
measure of maximal entropy with respect to $g_{N,\varepsilon}$. Consequently, for every sufficiently large \(N\),
there exists
\[
    \nu_N\in
    \mathcal M^{erg}\bigl(\Lambda(p,N),g_{N,\varepsilon}\bigr)
\]
such that
\[
    h_{\nu_N}(g_{N,\varepsilon})>
    \frac{\chi(p)}{r}-\frac{1}{n}.
\]
Moreover, by the construction of \(D_N\), for every \(0\le i\le k(N)\),
\[
    g_{N,\varepsilon}^i(\widetilde\Lambda_N)\subset U(p).
\]
Since \(k(N)\to+\infty\) as \(N\to+\infty\), whereas \(N_*\) is independent of
\(N\), the orbit of any point in \(\Lambda(p,N)\) spends an asymptotically
dominant proportion of its time inside \(U(p)\). Hence, taking \(N\) sufficiently
large, we obtain that for every \(z\in \Lambda(p,N)\),
\[
    \rho\bigl(\delta_p,\frac{1}{m}\sum_{i=0}^{m-1}\delta_{g_{N,\varepsilon}^iz}\bigr)<\frac{1}{2n}
\]
for all sufficiently large \(m\). In particular, this estimate holds for every
ergodic measure
\[
    \mu\in \mathcal M^{erg}\bigl(\Lambda(p,N),g_{N,\varepsilon}\bigr)
\]
and every generic point \(z\in G_\mu\).

Define
\[
    \mathcal R:=
    \bigcup_{m\in\mathbb N^+}\bigcap_{n\in\mathbb N^+}\mathcal S_n(p_{f_m})
,
\]
then \(\mathcal R\) is residual in \(\mathcal N\). Finally, for $f\in\mathcal R$, there exists $s\in\mathbb N^+$, such that
\[
    f\in\bigcap_{n\in\mathbb N^+}\mathcal S_n(p_{f_s}),\:p_{f_s}\in\mathcal H_{*}(f_s).
\]
Therefore, for each \(n\ge \operatorname{Per}(p_{f_s})\), there exists an ergodic
measure
\[
    \mu_n\in
    \mathcal M^{erg}\bigl(\Lambda(p_{f_s},n),f\bigr)
    \subset
    \mathcal M^{erg}(\mathbf M^2,f)
\]
such that
\[
    \liminf_{n\to+\infty} h_{\mu_n}(f)
    \ge
    \frac{\chi(p_{f_s})}{r}
    >0,
\]
$$\mu_n\to\delta_{p_{f_s}}^{\mathrm{per}(p_{f_s})},\:\mathrm{as}\:n\to+\infty.$$
By Corollary \ref{Corollary A}, Theorem \ref{Corollary E} has been proved up to now.
\subsection{Proof of Theorem C}
 Let  $D=\left\{(x,y)\in\mathbb R^2:x^2+y^2\le 4\right\}$.
Firstly, we state the result in \cite{B14}.
\begin{theorem}\cite[Buz14, Theorem 1]{B14}\label{Theorem 4.1}
 There exists $f\in\mathrm{Diff}^{\infty}(D)$ with $h_{\mathrm{top}}(f)=0$ and the following properties.
For any $1<r<+\infty$ and any neighborhood $U_0$ of $f$ in $\mathrm{Diff}^r(D)$, there exists $f_0\in U_0$ such that:\\
(1)\:$h_{\mathrm{top}}(f_0)=\frac{\lambda^+(f)}{r}>0$,\\
(2)\:$\sup_{\mu\in\mathcal M(D,f_0)}\lambda_1(f_0,\mu)=\frac{\lambda^+(f)}{r}$  and this supremum is not achieved,\\
(3)\:$f_0$ has no measure of maximal entropy.    
\end{theorem}
 
 If $d=2$, by Theorem \ref{Theorem 4.1} and Theorem \ref{Theorem A}, for any $1<r<+\infty$, there exists a sequence $\left\{f_{n}\right\}_{n\in\mathbb N^+}$ of $C^r$ diffeomorphisms\:(on $D$) converging $C^r$ to $f\in\mathrm{Diff}^\infty(D)$ with $h_{\mathrm{top}}(f)=0$ and a sequence $\left\{\mu_{n}\right\}_{n\in\mathbb N^+}$ of $\left\{f_{n}\right\}_{n\in\mathbb N^+}$-ergodic measures converging to an $f$-invariant measure $\mu$, such that
$$\lim_{n\to+\infty}h_{\mu_{n}}(f_{n})=h_{\mu}(f)+\frac{\lambda_1^+(f,\mu)}{r},$$
$$\lim_{n\to+\infty}\lambda_1^+(f_{n},\mu_n)=\frac{\lambda^+(f)}{r}<\lambda^+_{1}(f,\mu)=\lambda^+(f):=\lim_{n\to+\infty}\frac{1}{n}\log^+\left\|Df^n\right\|_{0}.$$
 Moreover, for any $n\in\mathbb N^+$, $h_{\mu_n}(f_n)\ge \frac{\lambda^+_{1}(f,\mu)}{r}-\frac{1}{n}$. This implies that $$\limsup_{n\to+\infty}d^u_{\max}(f_n,\mu_n)=1,$$
$$\lim_{n\to+\infty}h_{\mu_{n}}(f_{n})=h_{\mu}(f)+\frac{1}{r-1}(\lambda_1^+(f,\mu)-\lim_{n\to+\infty}\lambda_1^+(f_n,\mu_n)).$$

If $d=2m$, by considering $$\mathrm{Diff}^r(D^{m})\ni \underbrace{f_{n}\times\cdots\times f_{n}}_{m}\xrightarrow{C^r} \underbrace{f\times\cdots\times f}_{m}\in \mathrm{Diff}^\infty(D^m),$$ 
and 
$$\mathcal M(D^m,\underbrace{f_{n}\times\cdots\times f_{n}}_{m})\ni \underbrace{\mu_{n}\times\cdots\times \mu_{n}}_{m}\to \underbrace{\mu\times\cdots\times \mu}_{m}\in \mathcal M(D^m,\underbrace{f\times\cdots\times f}_{m}),$$ 
as $n\to+\infty$, we have
$$\limsup_{n\to+\infty}h_{\underbrace{\mu_{n}\times\cdots\times \mu_{n}}_{m}}(\underbrace{f_{n}\times\cdots\times f_{n}}_{m})= h_{\underbrace{\mu\times\cdots\times \mu}_{m}}(\underbrace{f\times\cdots\times f}_{m})+m\frac{\lambda_1^+(f,\mu)}{r}.$$
By the definitions of $(D^m,\underbrace{f\times\cdots\times f}_{m})\:\mathrm{and}\:\underbrace{\mu\times\cdots\times \mu}_{m}$, we have
$$\lambda^+(f)=\lambda^+_{1}(f,\mu)\le \lambda^+_{1}(\underbrace{f\times\cdots\times f}_{m},\underbrace{\mu\times\cdots\times \mu}_{m})\le \lambda^+(\underbrace{f\times\cdots\times f}_{m})=\lambda^+(f).$$
For $n\in\mathbb N^+$ large enough, we have
$$d_{\mathrm{max}}^u(f_n,\mu_n)=1.$$
It implies that
$$d_{\mathrm{max}}^u(\underbrace{f_n\times\cdots\times f_n}_{m},\underbrace{\mu_n\times \cdots\times \mu_n}_{m})=m.$$
Therefore,
\begin{equation}
\begin{aligned}
&\limsup_{n\to+\infty}h_{\underbrace{\mu_{n}\times\cdots\times \mu_{n}}_{m}}(\underbrace{f_{n}\times\cdots\times f_{n}}_{m})\\
&= h_{\underbrace{\mu\times\cdots\times \mu}_{m}}(\underbrace{f\times\cdots\times f}_{m})\\
&+\limsup_{n\to+\infty}d^u_{\max}(\underbrace{f_{n}\times\cdots\times f_{n}}_{m},\underbrace{\mu_n\times \cdots\times \mu_n}_{m})\frac{\lambda^+_{1}(\underbrace{f\times\cdots\times f}_{m},\underbrace{\mu\times\cdots\times \mu}_{m})}{r}.
\end{aligned}
\end{equation}
Similarly, one has
$$\lim_{n\to+\infty}h_{\mathrm{top}}(\underbrace{f_{n}\times\cdots \times f_n}_{m})= h_{\mathrm{top}}(\underbrace{f\times\cdots \times f}_{m})+m\frac{\lambda^+(\underbrace{f\times\cdots \times f}_{m})}{r}.$$

It is similar to prove the case of $d=2m+1$, by considering $$\mathrm{Diff}^r(\mathbb S^1\times D^{m})\ni \theta\times\underbrace{f_{n}\times\cdots\times f_{n}}_{m}\xrightarrow{C^r} \theta\times\underbrace{f\times\cdots\times f}_{m}\in \mathrm{Diff}^\infty(\mathbb S^1\times D^m),$$ and 
\begin{equation*}
\begin{aligned}
&\mathcal M(\mathbb S^1\times D^m,\theta\times \underbrace{f_{n}\times\cdots\times f_{n}}_{m})\ni \mathrm{Leb}\times\underbrace{\mu_{n}\times\cdots\times \mu_{n}}_{m}\\
&\to\mathrm{Leb}\times\underbrace{\mu\times\cdots\times \mu}_{m}\in \mathcal M(\mathbb S^1\times D^m,\theta\times\underbrace{f\times\cdots\times f}_{m}),
\end{aligned}
\end{equation*}
as $n\to+\infty$, where $\theta$ denotes the irrational rotation on the circle.

It is similar to prove the part (2) of Theorem \ref{Theorem B}, by considering $$\mathrm{Diff}^r(D\times\mathbb S^{d-2})\ni f_n\times\underbrace{\theta\times\cdots\times \theta}_{d-2}\xrightarrow{C^r} f\times\underbrace{\theta\times\cdots\times \theta}_{d-2}\in \mathrm{Diff}^\infty(D\times\mathbb S^{d-2}),$$ and 
\begin{equation*}
\begin{aligned}
&\mathcal M(D\times\mathbb S^{d-2},f_n\times\underbrace{\theta\times\cdots\times \theta}_{d-2})\ni \mu_n\times\underbrace{\mathrm{Leb}\times\cdots\times \mathrm{Leb}}_{d-2}\\
&\to\mu\times\underbrace{\mathrm{Leb}\times\cdots\times \mathrm{Leb}}_{d-2}\in \mathcal M(D\times\mathbb S^{d-2},f\times\underbrace{\theta\times\cdots\times \theta}_{d-2}),
\end{aligned}
\end{equation*}
as $n\to+\infty$, where $\theta$ denotes the irrational rotation on the circle.
\bigskip
\section*{Declarations}

\noindent\textbf{Funding.} Wanshan Lin is supported by the National Natural Science Foundation of China (No. 124B2010). Xueting Tian is supported by the National Natural Science Foundation of China (No. 12471182).

\noindent\textbf{Conflict of interest.} On behalf of all authors, the corresponding author states that there is no conflict of interest.

\noindent\textbf{Data availability.} On behalf of all authors, the corresponding author states that the manuscript has no associated data.


\newpage
\begin{thebibliography}{99}

\bibitem{B72} Rufus Bowen, Entropy-expansive maps, Trans. Amer. Math. Soc. 164 (1972), 323–331. MR 285689

\bibitem{Bu08}David Burguet, A proof of Yomdin-Gromov’s algebraic lemma, Israel J. Math. 168 (2008), 291–316. MR 2448063

\bibitem{B12}David Burguet, Symbolic extensions in intermediate smoothness on surfaces, Ann. Sci. Éc. Norm. Supér. (4) 45 (2012), no. 2,
337–362. MR 2977622

\bibitem{B17}David Burguet, Usc/fibred entropy structure and applications, Dyn. Syst. 32 (2017), no. 3, 391–409. MR 3669808

\bibitem{B23}David Burguet, Maximal measure and entropic continuity of Lyapunov exponents for $C^r$ surface diffeomorphisms with large
entropy, Ann. Henri Poincaré 25 (2024), no. 2, 1485–1510. MR 4703423

\bibitem{BL21} David Burguet and Gang Liao, Symbolic extensions for 3-dimensional diffeomorphisms, J. Anal. Math. 145 (2021),
no. 1, 381–400. MR 4361910

\bibitem{B97}Jérôme Buzzi, Intrinsic ergodicity of smooth interval maps, Israel J. Math. 100 (1997), 125–161. MR 1469107

\bibitem{B14}Jérôme Buzzi, $C^r$ diffeomorphisms with no maximal entropy measure, Ergodic Theory Dynam. Systems 34 (2014), no.6, 1770-1793. MR 3272770 

\bibitem{BCS22}Jérôme Buzzi, Sylvain Crovisier, and Omri Sarig, Continuity properties of Lyapunov exponents for surface diffeomorphisms, Invent. Math. 230 (2022), no.2, 767-849. MR4493327

\bibitem{CY16} Yongluo Cao and Dawei Yang, On Pesin’s entropy formula for dominated splittings without mixed behavior, J. Differ-
ential Equations 261 (2016), no. 7, 3964–3986. MR 3532061

\bibitem{DN05} T. Downarowicz and S.E. Newhouse, Symbolic extensions and smooth dynamical systems, Invent. Math. 160 (2005), no.3, 453-499. MR2178700


\bibitem{P12} Lorenzo J. Díaz, Todd Fisher, Maria José Pacifico, and José L. Vieitez, Entropy-expansiveness for partially hyperbolic
diffeomorphisms, Discrete Contin. Dyn. Syst. 32 (2012), no. 12, 4195–4207. MR 2966742

\bibitem{LY85-2} F. Ledrappier and L.-S. Young, The metric entropy of diffeomorphisms. II. Relations between entropy, exponents and
dimension, Ann. of Math. (2) 122 (1985), no. 3, 540–574. MR 819557

\bibitem{LSW15}
Gang Liao, Wenxiang Sun, and Shirou Wang, Upper semi-continuity of entropy map for nonuniformly hyperbolic
systems, Nonlinearity 28 (2015), no. 8, 2977–2992. MR 3382593


\bibitem{L13}
Gang Liao, Marcelo Viana, and Jiagang Yang, The entropy conjecture for diffeomorphisms away from tangencies, J.
Eur. Math. Soc. (JEMS) 15 (2013), no. 6, 2043–2060. MR 3120734

\bibitem{LMZ24} Chiyi Luo, Wenhui Ma, and Yun Zhao, Upper semi-continuity of metric entropy for diffeomorphisms with dominated splitting, (2025). arXiv:2412.04953. 

\bibitem{LY25-2} Chiyi Luo and Dawei Yang, Upper semi-continuity of metric entropy for $C^{1,\alpha}$ diffeomorphisms, (2025). arXiv: 2504.07746. 

\bibitem{N79}Sheldon E. Newhouse, The abundance of wild hyperbolic sets and nonsmooth stable sets for diffeomorphisms, Inst. Hautes Études Sci. Publ. Math. no.50 (1979), 101-151. MR2178700


\bibitem{N89}Sheldon E. Newhouse, Continuity properties of entropy, Ann. of Math. (2) 129 (1989), no. 2, 215–235. MR 986792

\bibitem{OB25}S. Ben Ovadia and David Burguet, Generalized u-Gibbs measures for $C^{\infty}$ diffeomorphisms, (2025). arXiv: 2506.18238.

\bibitem{PY76}Ja. B. Pesin, Families of invariant manifolds that correspond to nonzero characteristic exponents, Izv. Akad. Nauk
SSSR Ser. Mat. 40 (1976), no. 6, 1332–1379, 1440. MR 458490

\bibitem{PT93}J. Palis Jr. and F. Takens, Hyperbolicity and sensitive chaotic dynamics at homoclinic bifurcations, Cambridge Studies in Advanced Mathematics. 35, Cambridge Unvi. Press, Cambridge (1993). MR1237641


\bibitem{Y87} Y. Yomdin, Volume growth and entropy, Israel J. Math. 57 (1987), no. 3, 285–300. MR 889979
\end{thebibliography}
\end{document}